%% file: DraftQLAV3.tex
\documentclass[a4paper]{article}
\usepackage{dsfont}
\usepackage{amssymb}
\usepackage{latexsym}
\usepackage{amsmath}
\usepackage{color}
\usepackage{comment}
\usepackage{amsthm}
\usepackage[dvips]{graphicx}
\usepackage{layout} 
\usepackage{ulem}
\usepackage{enumerate}
\usepackage{bm}
\usepackage{bbm} 
\usepackage[bbgreekl]{mathbbol} 
\usepackage{authblk}
\usepackage[pdfborder={1 1 1}]{hyperref}  
\usepackage{multirow}
\newif\ifrs
\rstrue
\ifrs \usepackage{mathrsfs} \fi  
\newif\ifcol
\coltrue 

\newtheorem{theorem*}{Theorem}[section]

\newtheorem{note*}[theorem*]{Note}
\newtheorem{lemma*}[theorem*]{Lemma}
\newtheorem{definition*}[theorem*]{Definition}
\newtheorem{proposition*}[theorem*]{Proposition}
\newtheorem{corollary*}[theorem*]{Corollary}
\newtheorem{remark*}[theorem*]{Remark}
\newtheorem{example*}[theorem*]{Example}

\numberwithin{equation}{section}
\newif\ifcol
\coltrue
\ifcol
\newcommand{\colorr}{\color[rgb]{0.8,0,0}}

\newcommand{\colorn}{\color[rgb]{1,1,1}}

\else

\newcommand{\colorr}{\color{black}}

\newcommand{\colorn}{\color{black}}

\fi
\excludecomment{en-text}
\includecomment{jp-text}
\includecomment{comment}
\input nakamacro27.tex

\begin{document}

\title{Statistical Inference for Ergodic Point Processes and Application to Limit Order Book}
%

\author{Simon Clinet}
\author{Nakahiro Yoshida}
\affil{Graduate School of Mathematical Sciences, University of Tokyo
\footnote{Graduate School of Mathematical Sciences, University of Tokyo: 3-8-1 Komaba, Meguro-ku, Tokyo 153-8914, Japan. e-mail: simon@ms.u-tokyo.ac.jp, nakahiro@ms.u-tokyo.ac.jp}
        }
\affil{CREST, Japan Science and Technology Agency
        }       

\maketitle
\ \\
%
\begin{abstract} 
We construct a general procedure for the Quasi Likelihood Analysis applied to a multivariate point process on the real half line in an ergodic framework. More precisely, we assume that the stochastic intensity of the underlying model belongs to a family of processes indexed by a finite dimensional parameter. When a particular family of laws of large numbers applies to those processes, we establish the consistency, the asymptotic normality and the convergence of moments of both the Quasi Maximum Likelihood estimator and the Quasi Bayesian estimator. In addition, we illustrate our main results by showing how they can be applied to various Limit Order Book models existing in the literature. In particular, we address the fundamental cases of Markovian models and exponential Hawkes process-based models. 
\end{abstract}  

\section{Introduction}

Most financial transactions take place nowadays in electronic markets. In the so-called continuous-time double auctions system, traders can freely send buying or selling orders at any price level. As the matching process of those asynchronous orders is rather complex, every submission is centralized in an Electronic Limit Order Book (also denoted by LOB), waiting to be executed according to its price and time priority. An LOB is thus a multidimensional queuing system, that gathers the total volume of non-executed orders for every price level. Essentially, market agents interact with this dynamical system via three elementary rules. They may submit a buying (resp. selling) \textit{limit order} that will increase the size of one queue on the bid (resp. ask) side of the LOB. They also may send a buying (resp. selling) \textit{market order} that will immediately consume the corresponding liquidity at the best available price. Finally they can submit \textit{cancellation orders} to remove one of their latent limit order in the LOB. Driven by these simple mechanisms, the characteristics of the Order Book, such as its mid price, its shape, or the number of orders submitted in a given time window are subject to random fluctuations as time passes. \\  
\medskip

As the macroscopic price movements of an asset are determined by the evolution of its Order Book through time, understanding the stochastic structure of this object is a fundamental issue. It is also a way to describe in a high-dimensional context the microstructure phenomena related to the stock price movements. Indeed, phenomena such as the Epps effect, traditionally represented by an unknown noise process in many studies that confine themselves to one-dimensional price models, are captured when the whole Limit Order Book is taken into account. While the availability of high-frequency financial data has made in-depth analysis of LOB dynamics accessible, it is now possible to use statistical tools in order to highlight either empirical facts about their shapes and their evolution \cite{MikeFarmerEmpirical2008,BouchaudFarnerHowMarkets2008}, or to design and select models that are able to reproduce some of those stylized facts. For the latter, modelling the stochastic behavior of Limit Order Books has been the subject of an intense research over the last decade, the simplest approaches being zero-intelligence models, in which interarrival times of orders are exponentially distributed. The seminal work about this basic representation \cite{SmithFarmerGillemotKrishnamurthy2003} has been followed by many extensions, see e.g. \cite{AbergelLOB2013,ContLarrardDynamics2012,ContStoikovTalrejaStochastic2010,ContLarrardPrice2013,MuniTokeQueuing2013}. Despite their mathematical tractability, such models are too simple to be entirely satisfactory, essentially because of the absence of any structure of dependence in the order submission process. Lately more complex dynamics have been proposed, including Markovian models \cite{RosenbaumQueue2014}, Hawkes process driven bid-ask prices \cite{BacryMuzyHoffmannHawkes2013,ZhengRoueffAbergel2014}, or even Hawkes processes-based LOB models \cite{AbergelHawkes2015,muniToke2010,BacryJaisson2015}, the latter being reputed to efficiently capture the time clustering property of orders, and the cross-dependences in the order flow. \\      
\medskip

The fact that an LOB is mechanically driven by the orders submitted over time has encouraged many authors to see Order Books through the stochastic structure of their interarrival times $(\Delta T_0, \Delta T_1, ...)$ between two successive events. Accordingly, it seems appropriate to describe a Limit Order Book as a high-dimensional point process, that is a process $N$ whose components $N^\alpha_{t} - N^\alpha_{t^{'}} $ count the number of orders of same type, and at the same price level, that have been successively submitted in a time interval $[t,t^{'}]$. In a parametric context, estimating the true parameter $\theta^{*}$ of the model based on the observations is therefore a very general and crucial issue that may take place in (at least) two distinct asymptotics. As, for liquid stocks, a tremendously large number of events happen during short periods of times, the heavy traffic limit (very large number of events on a fixed time window) offers a favorable framework for the construction of consistent estimators. In \cite{YoshidaOgiharaQLA2015}, a sequence of multivariate point processes is thus assumed to be observable on a time interval $[T_0, T_1]$. Under suitable assumptions on the sequence of stochastic intensities itself, it is shown that even in this non-ergodic context it is possible to conduct the so-called Quasi Likelihood Analysis procedure (QLA). In particular, both the Quasi Maximum Likelihood estimator (QMLE) and the Quasi Bayesian estimator (QBE) are consistent, asymptotically mixed normal, and their moments converge.\\
\medskip

In this work, instead, we focus on the long run characteristics of the LOB seen as a point process. While the time parameter $T$ tends to infinity, assuming that the LOB satisfies a suitable ergodicity assumption, we make use of this property to derive the asymptotic properties of the QMLE and the QBE. This problematic is by no means new since the consistency and the asymptotic normality of the Maximum Likelihood estimator (MLE) for ergodic stationary point processes were established a few decades ago in \cite{OgataAsymptotic1978} and \cite{PuriTuanMLE1986}. \cite{tuan1981estimation} has also suggested a different approach based on parameter estimation for the spectral density of the process. In the non-stationary regime, the inhomogeneous Poisson process case has been deeply investigated in \cite{kutoyant1998poisson}, and conditions for the convergence of moments of the MLE and the Bayesian estimator (BE) of a general point process have been stated in \cite{kutoyants1984parameter} (see Theorems 4.5.5 and 4.5.6). Furthermore, Maximum Likelihood estimations have been empirically conducted for the abovementioned models, but the fact that such procedure is consistent is sometimes unclear. As a matter of fact, the term ergodicity itself can be subject to various definitions according to the underlying structure of the process (stationarity, Markovian property, ...). However, all those structures share an essential property, namely a family of laws of large numbers (LLN). Indeed, denoting by $\lambda(t,\theta)$ the stochastic intensity of the point process, we have in many cases

\bea
	\inv{T} \int_0^T{f(\lambda(s,\theta))ds} \to^\proba \pi(f,\theta),
	\label{llnIntro}
\eea
those convergences being the crucial assumptions to determine the asymptotic properties of both the QMLE and the QBE, regardless of the nature of the operator $\pi$. We therefore consider (\ref{llnIntro}) as a general definition of ergodicity in the present paper. We sometimes need to consider slightly more general processes than the simple vector $\lambda(t,\theta)$ in (\ref{llnIntro}) for technical considerations, although this is not necessary in many cases. Needless to say that this almost model-free representation covers thus a wide range of examples, and our main result is the derivation of the full Quasi Likelihood Analysis for point processes in this unified context. Under suitable assumptions that are given in this paper, both the QMLE and the QBE satisfy the following asymptotic properties : \\

\bea
\label{eqConsistency}
\theta_T &\to^\proba& \theta^{*}, \\ \label{eqNormality}
\sqrt T(\theta_T - \theta^{*}) &\to^d& \Gamma^{-\frac{1}{2}}\xi,\\ \label{eqmoments}
\esp[f(\sqrt T(\theta_T - \theta^{*}))] &\to& \esp[f(\Gamma^{-\frac{1}{2}} \xi)] \text{,  } \xi \sim  \mathcal{N},
\eea
where $\Gamma$ is the Fisher information, $\mathcal{N}$ is the standard normal distribution, and $f$ is any continuous function of polynomial growth. Here $\to^\proba$ and $\to^d$ respectively stand for the convergence in probability and the convergence in distribution. Let us give a few comments on the latter result. The convergence (\ref{eqmoments}) is essentially a consequence of the local asymptotic normality (LAN) of our statistical model at point $\theta^{*}$, along with some deviation inequality on the Quasi Likelihood random field. It is also well-known that if those properties hold true uniformly in $\theta^{*}$, then our estimators are asymptotically efficient in the sense of the H{\'a}jek-Le Cam bound. In other words, they attain the minimum of an asymptotic minimax bound for the risk associated to any loss function of polynomial growth. We refer the reader to \cite{hajek1972local} and \cite{ibragimov2013statistical}, Sections II.11-12 and III.1-2 for precise statements of these results. Finally, let us mention that the convergence of moments is very often needed in order to identify the bias correction term for the construction of information criteria. A presentation of information criterion-based model selection can be found in the introduction of \cite{uchida2001information}.  \\
\medskip

From a practical point of view, such general framework provides solid results on two very popular estimators that are commonly used in empirical works. Despite the fact that their properties were partially known in various cases, we believe that this unified view will both consolidate those results by systematically ensuring (\ref{eqConsistency})-(\ref{eqmoments}) and help the practitioner to easily validate the theoretic properties of her estimators when she is confronted with point process regression. Indeed, the central condition for our result is the LLN (\ref{llnIntro}), which can often be naturally derived from the underlying structure of the model she is willing to use. On the other hand, the question of the choice of the model, that also arises in practice is more complex. Consequently, further considerations, such as model selection, mispecification, or inference for more general processes are beyond the scope of this paper but are quickly mentioned in the conclusion. We finally insist on the fact that such general QLA procedure is applicable to any context involving an ergodic point process. Accordingly, the range of applications does not reduce to finance, but extends to diverse fields such as seismology and optics (see the introduction of \cite{kutoyant1998poisson}), or neurobiology (see Examples 2 and 3 in \cite{BremaudStability1996}) to name a few.  \\

\medskip

The paper is organized as follows. We introduce a general ergodic multivariate point process model in Section 2. Section 3 is devoted to show, under a classical approach first, the consistency and the asymptotic normality of the QMLE under suitable regularity conditions. These results are then generalized to the convergence of moments of both the QMLE and the QBE under stronger assumptions. We give some applications and examples that are picked from the literature on Limit Order Books and that satisfy those general conditions in Section 4, and we conclude this paper in Section 5. Finally, we gather some annex proofs and lemmas in Section 6.  \\

\pagebreak
\section{Multivariate point process}

We first introduce the general framework for a parametrized multivariate point process. Let $\calb=(\Omega,\calf,\F,\proba)$, $\F=(\calf_t)_{t\in\bbR_+}$ be a stochastic basis, and assume that we are given a multidimensional point process $N_t = (N_t^\alpha)_{\alpha \in \mathbf{I}}$, $\mathbf{I} = \{1,...,d\}$ that is adapted to $\F$. For example, each $N_t^\alpha$ can be thought as the number of limit (resp. cancellation, market) orders at some given price level $p_\alpha$ that have been submitted in the time interval $[0,t]$. The general definition of a Limit Order Book and the multivariate point process associated to it will be given in Section 4, since no Limit Order Book knowledge is needed to understand this statistical part. The filtration $\F = (\calf_t)_{t \reels_+}$ is generated by the collection of observable processes involved in the structure of $N$. In the most basic case $\F$ is generated solely by $N$ itself and is thus the so-called canonical filtration of the point process, as is the case for the self-exciting Hawkes process as defined in Section 4.2. On the other hand, Examples 3.3, 3.4 and 3.6 along with the LOB models described in Examples 4.2 and 4.3 include additional observed explanatory processes. \\
\medskip

 For the sake of simplicity $N$ is assumed to be defined on $\reels_+$, $N=(N_t)_{t \in \reels_+}$, $N_0=0$, and its components $N^\alpha$ have no common jumps. $N$ can also be seen as an integer measure if we write for any Borel subset $A$, $N(A)= \int_A{dN_s}$. In particular $N_t = N([0,t])$. Understanding the dynamic of a point process on $\reels_+$ can be done, as usual, by studying its underlying stochastic intensity. We recall that if we write $\Lambda_t$ the $\calf_t$-compensator of $N_t$, we call $\lambda$ the (predictable) $\calf_t$-stochastic intensity of $N$, that is the unique (predictable) process such that

\beas
\Lambda_t = \int_0^t{\lambda(s)ds}
\eeas
if such representation exists. Intuitively speaking, $\lambda(t)dt$ is, conditionally to $\calf_t$, the average (multivariate) number of jumps $dN_t$, or in an informal way, $\esp \l[dN_t | \calf_{t} \r] =  \lambda(t)dt$.\\

\medskip

Since $\lambda$ determines the law of the process $N$, the choice of the form of $\lambda$ is an important question. In a parametric context, it is common to allow the stochastic intensity itself to depend on a parameter $\theta$ as follows : for some finite dimensional relatively compact open space $\Theta \subset \reels^n$, $n \geq 1$, we consider the family of adapted (and thus observable) processes $\Lambda_t(\theta) = \int_0^t{\lambda(s,\theta)ds}$. We assume that there exists an unknown \textit{true value} $\theta^{*} \in \Theta$ such that $\lambda(t,\theta^{*})$ is the actual $\calf_t$-intensity of $N_t$. The process $\tilde{N}_t=N_t - \Lambda_t(\theta^{*})$ is thus a local martingale. \\

\medskip

Let us introduce a few terms and notations. All along the paper, if $x$ designates a real number, a vector or a matrix, $|x| = \sum_i |x|_i$. For a vector $x \in \reels^n$, $x^{\otimes 2}$ stands for the product of $x$ and its transpose : $x^{\otimes 2}=x.x^T \in \reels^{n \times n}$.  If $X$ is a random variable, $\|X\|_p = \esp \l[ |X|^p \r]^{\frac{1}{p}}$. If $a_T$ and $b_T$ are sequences of random variables, $a_T = o_\proba(b_T)$ stands for $\frac{a_T}{b_T}1_{\{b_T \neq 0\}} \to^\proba 0$. We also write $a_T = O_\proba(b_T)$ to say that $\frac{a_T}{b_T}1_{\{b_T \neq 0\}}$ is stochastically bounded. For a process $X$, $\calf_t^X = \sigma\{X_s, 0 \leq s \leq t\} $ designates the canonical filtration of $X$. Finally, for a Borel space $(E,\mathbf{B}(E))$, $C_b(E,\reels)$ is the set of continuous, bounded functions from the space $E$ to $\reels$, equipped with their underlying topologies.\\

\medskip

We will work with the following two fundamental conditions on the family $\l(\lambda(t,\theta)\r)_{t \in \reels_+,\theta \in \Theta}$ :

\bd
\im[[A1\!\!]]
The mapping $\lambda : \Omega \times \reels_+ \times \Theta \to \reels_+$ is $\calf \otimes \mathbf{B}(\reels_+) \otimes \mathbf{B}(\Theta)$-measurable. Moreover, almost surely,
  \bd
		\im[{\bf (i)}] for any $\theta \in \Theta$, $s \to \lambda(s,\theta)$ is left continuous.
		\im[{\bf (ii)}] for any $s \in \reels_+$, $\theta \to \lambda(s,\theta)$ is in $C^3(\Theta)$, and admits a continuous extension to $\overline{\Theta}$. 
	\ed

\im[[A2\!\!]]
The intensity processes and their first derivatives satisfy 
	\bd
		\im[(i)] for any $p > 1$, $\sup_{t \in \reels_+}  \sum_{i=0}^3{\left\| \sup_{\theta \in \Theta}\l|\partial_\theta^i \lambda(t,\theta)\r|\right\|_p} <+\infty.$
		\im[(ii)] for any $p > 1$, for any $\alpha \in \mathbf{I}$, $\sup_{t \in \reels_+} \left\| \sup_{\theta \in \Theta}\l|\lambda^\alpha(t,\theta)^{-1}\r| 1_{\{\lambda^\alpha(t,\theta) \neq 0\}}\right\|_p <+\infty.$
		\im[(iii)] For any $\theta \in \Theta$, for any $\alpha \in \mathbf{I}$, $\lambda^\alpha(t,\theta) = 0 $ if and only if  $\lambda^\alpha(t,\theta^{*}) =0$.
	\ed
\ed

\begin{remark}\rm 
 The processes $\sup_{\theta \in \Theta}\partial_\theta^i \lambda(.,\theta)$ are $\calf \otimes \mathbf{B}(\reels_+)$-measurable mappings. Indeed, the continuity in $\theta$ of $\lambda$ ensures that the $\sup$ can be taken over a countable dense subset $\{\theta_i | i \in \naturels\} \subset \Theta$. Note also that allowing intensities to vanish in [A2] is necessary in an LOB framework since in practice cancellation orders can occur at some level $\alpha \in \mathbf{I}$ only if the corresponding limit is not empty.
\end{remark}

 We introduce the Quasi-Log Likelihood process as

\bea
l_T(\theta) = \sum_{\alpha \in \mathbf{I}}{\int_0^T{\text{log}(\lambda^\alpha(s,\theta))dN_s^\alpha}-\sum_{\alpha \in \mathbf{I}}\int_0^T{\lambda^\alpha(s,\theta)ds}}.
\label{eqll}
\eea

Note that although $\lambda^\alpha(s,\theta)$ may vanish, (\ref{eqll}) is well defined as the non-negative process $\int_0^T{1_{\{\lambda^\alpha(s,\theta) = 0\}}dN_s^\alpha}$ verifies 

\bea \esp \left[ \int_0^T{1_{\{\lambda^\alpha(s,\theta) = 0\}} dN_s^\alpha}\right] =\esp \left[ \int_0^T{1_{\{\lambda^\alpha(s,\theta^{*}) = 0\}} \lambda^\alpha(s,\theta^{*})ds}\right] = 0.
\label{eqNoJumps}
\eea

Writing $\Delta N_s^\alpha $ for $ N_s^\alpha - N_{s-}^\alpha$, we notice that $\proba$-a.s., thanks to (\ref{eqNoJumps}), for any $s \in \reels_+$, $1_{\{\lambda^\alpha(s,\theta) = 0\}} \Delta N_s^\alpha = 0$. The previous equality entails the finiteness of $\text{log}(\lambda^\alpha(s,\theta))$ whenever a jump occurs, and therefore the integral in $dN_s^\alpha$ of the Quasi-Log Likelihood process is well-defined. In most situations, this process is indeed the real Log Likelihood process up to a constant term, as is the case when $\calf_t = \calf_t^N$ is the canonical filtration of $N$ (see Theorem 5.45, Chapter III in \cite{JacodLimit2003}).\\ 

\medskip

Finally we say that a statistic $\hat{\theta}_T$ is an asymptotic Quasi-Maximum Likelihood estimator if it asymptotically maximizes the rescaled Quasi-Log Likelihood process in the following sense :

\bea
 \frac{l_T(\hat{\theta}_T)}{T} \geq \sup_{\theta \in \Theta}{\frac{l_T(\theta)}{T}}- o_{\proba}\l(\inv{\sqrt T}\r).
\eea

We also call QMLE or exact QMLE the estimator that maximizes the rescaled Quasi-Log Likelihood. In section 3 we give a general ergodicity assumption, and prove the consistency and the asymptotic normality of any asymptotic QMLE. We finally derive the convergence of moments of both the exact QMLE and the QBE under a stronger version of the ergodicity assumption. A summary of the asymptotic properties of the estimators and the related assumptions can be found in Table \ref{tableSummary}.

\section{Quasi Likelihood Analysis}

\subsection{The ergodicity assumption}

We make an extensive use of the law of large numbers in our proofs. In particular the Fisher information matrix and the asymptotic Quasi-Log Likelihood can be expressed as such time average limits. In the sequel, we refer to the following general definition when we consider ergodic processes :
\begin{definition*}
Let $(E,\mathbf{B}(E))$ be a Borel space, and $X : \Omega \times \reels_+ \to E$ a stochastic process. We say that $X$ is ergodic if there exists a mapping $\pi : C_b(E,\reels) \to \reels $ such that for any $\psi \in C_b(E,\reels)$ 

\beas
\inv{T} \int_0^T{\psi(X_s)ds} \to^\proba \pi(\psi).  
\eeas
\label{defErgo}
\end{definition*}
 The main assumption that is at the core of our results is the following ergodicity condition on the processes $(\lambda^\alpha(.,\theta^{*}),\lambda^\alpha(.,\theta),\partial_\theta \lambda^\alpha(.,\theta))$ taking values in $ E = \reels_+ \times \reels_+ \times \reels^n $ : 

\bd
\im[[A3\!\!]] 
 For any $\alpha \in \mathbf{I}$, $\theta \in \Theta$, the triplet $(\lambda^\alpha(.,\theta^{*}),\lambda^\alpha(.,\theta),\partial_\theta \lambda^\alpha(.,\theta))$ is ergodic in the sense of Definition \ref{defErgo}. In other words, there exists a mapping $\pi_\alpha : C_b(E,\reels) \times \Theta \to \reels$ such that for any  $(\psi,\theta) \in C_b(E,\reels) \times \Theta $, the following convergence holds :
	$$\frac{1}{T}\int_0^T{\psi(\lambda^\alpha(s,\theta^{*}),\lambda^\alpha(s,\theta),\partial_\theta \lambda^\alpha(s,\theta))ds} \to^\proba \pi_\alpha(\psi,\theta). $$

\ed

Every component of the process $(\lambda(.,\theta^{*}),\lambda(.,\theta),\partial_\theta \lambda(.,\theta))$ is therefore assumed to be ergodic. Before we turn to the statement of our main results, let us raise a few examples to illustrate this property in various situations. 

\begin{example*}
\label{ex1}
Assume that $N$ is a univariate inhomogeneous Poisson process, that is a point process whose intensity is a deterministic function. If $t \to \lambda(t,\theta)$ is periodic of period $\tau \in \reels_+ - \{0\}$, then \textnormal{[A3]} holds with $\pi(f,\theta) = \inv{\tau}\int_0^\tau{f(\lambda(s,\theta^{*}),\lambda(s,\theta),\partial_\theta \lambda(s,\theta))ds}$. \textnormal{[A3]} also holds if for $i \in \{0,1\}$, $\partial_\theta^i \lambda(t,\theta) \to \partial_\theta^i\lambda_\infty(\theta)$ as $t \to \infty$ with $\pi(f,\theta) = f(\lambda_\infty(\theta^{*}),\lambda_\infty(\theta),\partial_\theta \lambda_\infty(\theta))$.
\end{example*} 

\begin{example*}
Assume that $N$ counts the jumps of an observable Pure Jump-Type Markovian process $Y$ on a discrete state space (see \cite{KallenbergFoundation2002}, chapter 12). It is a straightforward consequence of the definition of $Y$ and $N$ that the $\calf_t^Y$-stochastic intensity has the representation $\lambda(t,\theta) = g(Y_{t-},\theta)$ for some function $g$. Let us assume further that $Y$ is ergodic (see \cite{MeynTweedieprocessii1993} for a presentation of ergodicity for Markovian processes). If $Y$ takes values in $\cal Y$ and has invariant probability $\pi^Y$, \textnormal{[A3]} holds with 
$\pi(f,\theta) = \int_{\cal Y}{f(g(y,\theta^{*}),g(y,\theta),\partial_\theta g(y,\theta))\pi^Y(dy)}$. 
\label{purejumpex}
\end{example*}

\begin{example*}
\label{ex2}
Let us assume that $N$ is a univariate Cox process directed by some observable left-continuous ergodic Markovian process $Y$ (a rigourous definition of a Cox process can be found in \cite{KallenbergFoundation2002}, chapter 12). Accordingly, assume that the $\calf_t^{(N,Y)}$-intensity has the form $\lambda(t,\theta) = g(Y_{t},\theta)$ for some continuous function $g$. The same conclusion as in Example \ref{purejumpex} holds.
\end{example*}

\begin{example*}
\label{ex3}
Consider a univariate stationary point process $N$, that is a point process such that for any $r \in \naturels$ and any bounded Borel subsets $A_1,...,A_r$, the distribution of $(N(A_1+t),...,N(A_r+t))$ does not depend on $t$. The $\calf_t^N$-intensity of $N$ is thus automatically a stationary process. Let us assume that $N$ is stationary ergodic, meaning that its invariant $\sigma$-field is trivial. It is straightforward to see that \textnormal{[A3]} holds with $\pi(f,\theta) = \esp \l[f(\lambda(0,\theta^{*}),\lambda(0,\theta),\partial_\theta \lambda(0,\theta)) \r]$, by virtue of Birkhoff's ergodic Theorem.   
\end{example*}

\begin{example*}
\label{ex4}
Take the same representation as in Example \ref{ex2}, but instead of the Markovian property, assume that the process $(Y_t)_t$ has continuous paths and enjoys the regenerative property (see \cite{AsmussenAppliedProbability2003} chapter VI). Call $\tau_1,\tau_2,...$ the (random) regenerative times of $Y$ and assume that $\mu = \esp[\tau_1] < \infty$. Then the ergodicity condition \textnormal{[A3]} is satisfied with $\pi(f,\theta) = \inv{\mu} \esp\int_0^{\tau_1}{f(g(Y_s,\theta^{*}),g(Y_s,\theta),\partial_\theta g(Y_s,\theta))ds}$.
\end{example*}

We will present examples that are directly related to LOB modelling in Section 4. Before we turn to the classical approach, let us state a few useful results about the mapping $\pi$ and the class of functions that satisfy [A3]. Such extension is necessary because variables of interest such as the Fisher information are of the form $\pi(f,\theta)$ for unbounded functions $f$ that have singularities when one of their argument vanishes. Using [A2], it is indeed possible to extend the range of functions for which [A3] holds, and show that each $\pi_\alpha(\psi,\theta)$ can be seen as the integral of $\psi$ with respect to an actual probability measure $\pi_\alpha^\theta$ on $E$.

\begin{definition*}
Recall that $E = \reels_+ \times \reels_+ \times \reels^d$. We write $C_\uparrow(E,\reels)$ the set of functions $\psi : (u,v,w) \to \psi(u,v,w)$ from $E$ to $\reels$ that satisfy :
\begin{itemize}
\im $\psi$ is continuous on $\l(\reels_+ - \{0\} \r) \times \l(\reels_+ - \{0\} \r)\times \reels^d$.
\im $\psi$ is of polynomial growth in $(u,v,w,\frac{1_{\{u >0\}}}{u},\frac{1_{\{v >0\}}}{v})$.
\im For any $(u,v,w) \in E$, $\psi(0,v,w) = \psi(u,0,w) = 0$.
\end{itemize}

\end{definition*}

\begin{proposition*}
\label{extensionA3}
 For a measure $\mu$, let $\mathbb{L}^1(\mu)$ be the space of functions that are integrable with respect to $\mu$. Then for any $\theta \in \Theta$ and for any $\alpha \in \mathbf{I}$, the following properties hold.

\bd
\im [(i)]The law of large numbers stated in \textnormal{[A3]} still holds for any $\psi \in C_\uparrow(E,\reels)$. In particular, the mapping $\pi_\alpha(.,\theta)$ can be extended to $C_\uparrow(E,\reels)$. Moreover, for any $\psi \in C_\uparrow(E,\reels)$ the convergence is uniform in $\theta$.
\im [(ii)] There exists a probability measure $\pi_\alpha^\theta$ on $(E, \mathbf{B}(E))$ such that for any $\psi \in C_\uparrow(E,\reels)$, $\pi_\alpha(\psi,\theta) = \int_{E}{\psi(x)\pi_\alpha^\theta(dx)}$. In particular, $C_\uparrow(E,\reels) \subset \mathbb{L}^1(\pi_\alpha^\theta)$. 
\ed
\end{proposition*}

We are now ready to investigate the asymptotic properties of our estimators.


\subsection{A classical approach for the QMLE}

From now on we adopt a martingale approach to derive the consistency and the asymptotic normality of any asymptotic QMLE $\hat{\theta}_T$. The main results are stated in Theorem \ref{thmConsistency} and Theorem \ref{thmAsymptoticNormality}. Consider 
\beas
\mathbb{Y}_T(\theta) = \frac{1}{T}(l_T(\theta)-l_T(\theta^{*}))
\eeas
and the asymptotic rescaled Quasi-Log Likelihood 
\bea
\mathbb{Y}(\theta)  = \sum_{\alpha \in \mathbf{I}} \int_{(u,v,w) \in E}{1_{\{u > 0,　 v > 0 \}}\l\{\text{log} \l(\frac{v}{u}\r)u - (v-u)\r\}\pi_\alpha^\theta(du,dv,dw)}.
\eea
As we shall see in this section, $\mathbb{Y}(\theta)$ is the limit of $\mathbb{Y}_T(\theta)$ uniformly in $\theta \in \Theta$. The non-degeneracy of $\mathbb{Y}$ is as usual a crucial point to derive the consistency of our estimator. One possible formulation is the following :

\bd
\im[[A4\!\!]] For any $\theta \in \Theta - \{\theta^{*}\}$, $\mathbb{Y}(\theta) \neq 0$.
\ed


\begin{theorem*}
	Under \textnormal{[A1]-[A4]}, any asymptotic QMLE $\hat{\theta}_T$ is consistent.
		\beas
		\hat{\theta}_T \to^\proba \theta^{*}.
		\eeas
	\label{thmConsistency}
\end{theorem*}

We first deal with the uniform convergence of $\mathbb{Y}_T$ to $\mathbb{Y}$.

\begin{lemma*}
Under \textnormal{[A1]-[A3]},
 
$$ \sup_{\theta \in \Theta}{|\mathbb{Y}_T(\theta)-\mathbb{Y}(\theta)|} \to^\proba 0.$$

\end{lemma*}

\begin{proof}

	Define as previously the following local martingale :
		$$\tilde{N}_t^\alpha = N_t^\alpha - \int_0^t{\lambda^\alpha(s,\theta^{*})ds}.$$
	
		We rewrite $\mathbb{Y}_T(\theta)$ as 
		\begin{eqnarray*}
		\mathbb{Y}_T(\theta) &=& \frac{1}{T}\sum_{\alpha \in \mathbf{I}}{\int_0^T{\text{log}\frac{\lambda^\alpha(s,\theta)}{\lambda^\alpha(s,\theta^{*})}1_{\{\lambda^\alpha(s,\theta^{*}) \neq 0\}}d\tilde{N}_s^\alpha}}\\
		&-&\frac{1}{T}\sum_{\alpha \in \mathbf{I}}{\int_0^T{\left[\lambda^\alpha(s,\theta)-\lambda^\alpha(s,\theta^{*})-\text{log}\frac{\lambda^ \alpha(s,\theta)}{\lambda^\alpha(s,\theta^{*})}\lambda^\alpha(s,\theta^{*})\right]1_{\{\lambda^\alpha(s,\theta^{*}) \neq 0\}}ds}},\\
		\end{eqnarray*}
		and for $\alpha \in \mathbf{I}$ we put 
		
		  $$M_T^{\alpha}(\theta) = \int_0^T{\text{log}\frac{\lambda^\alpha(s,\theta)}{\lambda^\alpha(s,\theta^{*})}1_{\{\lambda^\alpha(s,\theta^{*}) \neq 0\}}d\tilde{N}_s^\alpha}$$
			and 
			
			$$V_T^{\alpha}(\theta) = -\frac{1}{T}\int_0^T{\left[\lambda^\alpha(s,\theta)-\lambda^\alpha(s,\theta^{*})-\text{log}\frac{\lambda^\alpha(s,\theta)}{\lambda^\alpha(s,\theta^{*})}\lambda^\alpha(s,\theta^{*})\right]1_{\{\lambda^\alpha(s,\theta^{*}) \neq 0\}}ds}.$$
			
			
			
			Let us first show that the martingale term $M_T^\alpha(\theta)$ tends to zero uniformly in $\theta$. Thanks to [A2], $M_T^\alpha(\theta)$ is a well-defined $\mathbb{L}^p$ integrable martingale. We apply Sobolev's inequality (see e.g. \cite{KallenbergFoundation2002}, Theorem 4.2, Part I, case A, with $j=0$, $m=1$, and any $p>n$). We thus take some integer $p > n$ and some constant $K(\Theta,p)$ such that 
				
				\beas
					\esp \l|\sup_{\theta \in \Theta} \frac{M_T^\alpha(\theta)}{T} \r|^p &\leq& \frac{K(\Theta,p)}{T^p} \l(\int_{\Theta}{d\theta \esp\l|M_T^\alpha(\theta)\r|^p}+\int_{\Theta}{d\theta \esp\l|\partial_\theta M_T^\alpha(\theta)\r|^p} \r). 
				\eeas
				
				Now, applying successively Davis-Burkholder-Gundy's inequality, Jensen's inequality and assumption [A2] it is straighforward to see that for some constant $C >0$
				
				\beas
				\esp[(M_T^\alpha(\theta))^p] &\leq& C \esp \l(\int_0^T{\l(\text{log}\frac{\lambda^\alpha(s,\theta)}{\lambda^\alpha(s,\theta^{*})}\r)^2\lambda^\alpha(s,\theta^{*})1_{\{\lambda^\alpha(s,\theta^{*}) \neq 0\}}ds} \r)^{\frac{p}{2}} \\
				&\leq& C T^{\frac{p}{2}-1} \esp \int_0^T{\l(\text{log}\frac{\lambda^\alpha(s,\theta)}{\lambda^\alpha(s,\theta^{*})}\r)^p\l(\lambda^\alpha(s,\theta^{*}) \r)^{\frac{p}{2}}1_{\{\lambda^\alpha(s,\theta^{*}) \neq 0\}}ds} \\
				&=& O\l(T^{\frac{p}{2}}\r), 
				\eeas

				and 
				\beas 
				\esp[ (\partial_\theta M_T^\alpha(\theta))^p] &\leq& C\esp \l(\int_0^T{\l(\frac{\partial_\theta\lambda^\alpha(s,\theta)}{\lambda^\alpha(s,\theta)}\r)^2\lambda^\alpha(s,\theta^{*})1_{\{\lambda^\alpha(s,\theta^{*}) \neq 0\}}ds}\r)^{\frac{p}{2}}\\
				&\leq&C T^{\frac{p}{2}-1}\esp \int_0^T{\l(\frac{\partial_\theta\lambda^\alpha(s,\theta)}{\lambda^\alpha(s,\theta)}\r)^p\l(\lambda^\alpha(s,\theta^{*})\r)^{\frac{p}{2}}1_{\{\lambda^\alpha(s,\theta^{*}) \neq 0\}}ds}\\
				  &=& O\l(T^{\frac{p}{2}}\r),
				\eeas
				where the permutation of the symbol $\partial_\theta$ and $\int_0^T$ is permitted by Lemma \ref{deriv}. Hence $\esp \l|\sup_{\theta \in \Theta} \frac{M_T^\alpha}{T} \r|^p \to 0$. Finally, thanks to Proposition \ref{extensionA3},  $V_T(\theta) = \sum_{\alpha \in \mathbf{I}}V_T^\alpha(\theta)$ converges in probability to $\mathbb{Y}(\theta)$ uniformly in $\theta$. Thus 
				
				\beas
				 \sup_{\theta \in \Theta}{|\mathbb{Y}_T(\theta)-\mathbb{Y}(\theta)|} \to^\proba 0.
				\eeas

\end{proof}

\begin{remark}\rm
The Sobolev's inequality depends on regularity properties of the domain $\Theta$, usually expressed as geometric conditions. Such conditions can be found in \cite{adams2003sobolev}. We will assume the following sufficient condition that is used in \cite{YoshidaOgiharaQLA2015} as well : $\inf_{\theta \in \Theta}\text{Leb}(B(\theta,\epsilon) \cap \Theta) \geq a_0(\epsilon^n \wedge 1)$ for any $\epsilon > 0$, where $a_0$ is some positive constant,  $B(\theta,\epsilon)$ is the open ball centered on $\theta$ with diameter $\epsilon$, and Leb is the Lebesgue measure.  
\end{remark}

Finally the consistency easily follows :

\begin{proof}[Proof of Theorem \ref{thmConsistency}]
	from the expression 
	\beas
\mathbb{Y}(\theta)  = \sum_{\alpha \in \mathbf{I}} \int_{(u,v,w) \in E}{1_{\{u > 0,　 v > 0 \}}\l\{\text{log} \l(\frac{v}{u}\r)u - (v-u)\r\}\pi_\alpha^\theta(du,dv,dw)},
	\eeas
	we immediately deduce that $\mathbb{Y} \leq 0$ and $\mathbb{Y}(\theta^{*}) = 0$. By [A4], $\theta^{*}$ is thus a global maximum of $\mathbb{Y}$, and by the previous lemma this ensures the consistency of any asymptotic QMLE. 
\end{proof}

We now turn to the asymptotic normality of $\hat{\theta}_T$. Using again [A2] and a variant of Lemma \ref{deriv}, we define for any $\theta \in \Theta$, 

\beas 
\partial_\theta l_T(\theta) &=& \sum_{\alpha \in \mathbf{I}}{\int_0^T{\lambda^\alpha(s,\theta)^{-1}\partial_\theta \lambda^\alpha(s,\theta)1_{\{\lambda^\alpha(s,\theta^{*}) \neq 0\}}dN_s^\alpha}}-\sum_{\alpha \in \mathbf{I}}\int_0^T{\partial_\theta\lambda^\alpha(s,\theta)ds},
\eeas 
which, evaluated at point $\theta^*$, has the form
\beas
\partial_\theta l_T(\theta^{*}) &=& \sum_{\alpha \in \mathbf{I}}{\int_0^T{\lambda^\alpha(s,\theta^{*})^{-1}\partial_\theta \lambda^\alpha(s,\theta^{*})1_{\{\lambda^\alpha(s,\theta^{*}) \neq 0\}}d\tilde{N}_s^\alpha}}, 
\eeas
and 
\beas
\partial_{\theta}^2 l_T(\theta) &=&\sum_{\alpha \in \mathbf{I}}{\int_0^T{\partial_\theta\left(\lambda^\alpha(s,\theta)^{-1}\partial_\theta \lambda^\alpha(s,\theta)\right)1_{\{\lambda^\alpha(s,\theta^{*}) \neq 0\}}d\tilde{N}_t^\alpha}}\\
&-&\sum_{\alpha \in \mathbf{I}}{\int_0^T{(\partial_\theta \lambda^\alpha)^{\otimes 2}(s,\theta)\lambda^\alpha(s,\theta)^{-2}\lambda^\alpha(s,\theta^{*})1_{\{\lambda^\alpha(s,\theta^{*}) \neq 0\}}ds}}\\
&+&\sum_{\alpha \in \mathbf{I}}{\int_0^T{\partial_{\theta}^2 \lambda^\alpha(s,\theta)\lambda^\alpha(s,\theta)^{-1}(\lambda^\alpha(s,\theta)-\lambda^\alpha(s,\theta^{*}))}1_{\{\lambda^\alpha(s,\theta^{*}) \neq 0\}}ds}.\\
\eeas

Finally consider the Fisher information matrix 
\bea
\Gamma = \sum_{\alpha \in \mathbf{I}} \int_{(u,v,w) \in E}{w^{\otimes 2}\frac{1_{\{u >0\}}}{u} \pi_\alpha^{\theta^*}(du,dv,dw)} \in \reels^{n \times n}.
\eea

If $\Gamma$ is not singular, the asymptotic normality holds :

\begin{theorem*}
	Let $\hat{\theta}_T$ be an asymptotic QMLE and assume that $\Gamma$ is positive definite. We have 
	\beas
	\sqrt{T}(\hat{\theta}_T-\theta^{*}) \to^{d} \Gamma^{-\frac{1}{2}}\xi,
	\eeas
	where $\xi$ follows a standard normal distribution.
	\label{thmAsymptoticNormality}
\end{theorem*}

We have divided the proof of this result into the next two lemmas.

\begin{lemma*}

Under \textnormal{[A1]-[A3]}, if $V_T$ is a ball centered on $\theta^{*}$ shrinking to $\{\theta^{*}\}$, then 
$$ \sup_{\theta \in V_T} |T^{-1} \partial_{\theta}^2 l_T(\theta) + \Gamma| \to^\proba 0.$$
\label{Fisher}
\end{lemma*}

\begin{proof}
	We first deal with the martingale part as for the consistency. We can easily see that the processes
	$\partial_\theta\left(\lambda^\alpha(s,\theta)^{-1}\partial_\theta \lambda^\alpha(s,\theta)\right)1_{\{\lambda^\alpha(s,\theta^{*}) \neq 0\}}$ and $\partial_\theta^2\left(\lambda^\alpha(s,\theta)^{-1}\partial_\theta \lambda^\alpha(s,\theta)\right)1_{\{\lambda^\alpha(s,\theta^{*}) \neq 0\}}$ are dominated by polynoms in $\partial_\theta^i \lambda^\alpha(t,\theta)$ and $\lambda^\alpha(t,\theta)^{-1}1_{\{\lambda^\alpha(s,\theta^{*}) \neq 0\}}$ for $i \in \{0,1,2,3\}$, and thus by an immediate application of Sobolev's inequality and [A2], we get for any $p >n$
	\beas
		\esp \l| \sup_{\theta \in \Theta}\frac{1}{T}\int_0^T{\partial_\theta\left(\lambda^\alpha(s,\theta)^{-1}\partial_\theta \lambda^\alpha(s,\theta)\right)1_{\{\lambda^\alpha(s,\theta^{*}) \neq 0\}}d\tilde{N}_t^\alpha}\r|^p  &=& O\l(T^{-\frac{p}{2}}\r),\\
	\eeas
	thus $$\inv{T}\sup_{\theta \in \Theta} \sum_{\alpha \in \mathbf{I}}{\int_0^T{\partial_\theta\left(\lambda^\alpha(s,\theta)^{-1}\partial_\theta \lambda^\alpha(s,\theta)\right)1_{\{\lambda^\alpha(s,\theta^{*}) \neq 0\}}d\tilde{N}_t^\alpha}} \to^\proba 0. $$
	
	For $\theta \in V_T$, we have 
	
	\beas
	&&\inv{T}\l|\int_0^T{\partial_{\theta}^2 \lambda^\alpha(s,\theta)\lambda^\alpha(s,\theta)^{-1}(\lambda^\alpha(s,\theta)-\lambda^\alpha(s,\theta^{*}))1_{\{\lambda^\alpha(s,\theta^{*}) \neq 0\}}ds}\r|\\
	&\leq&\frac{|\theta - \theta^{*}|}{T} \int_0^T{ \sup_{\theta \in \Theta}\l|\partial_{\theta}^2 \lambda^\alpha(s,\theta)\lambda^\alpha(s,\theta)^{-1}\r|\sup_{\theta \in \Theta}|\partial_\theta \lambda^\alpha(s,\theta)|1_{\{\lambda^\alpha(s,\theta^{*}) \neq 0\}}ds}\\
	&=& O_\proba(|\theta - \theta^{*}|)\\
	&=& O_\proba(\text{diam}[V_T]),\\
	\eeas
	therefore $$\sup_{\theta \in V_T} \inv{T}\sum_{\alpha \in \mathbf{I}}{\int_0^T{\partial_{\theta}^2 \lambda^\alpha(s,\theta)\lambda^\alpha(s,\theta)^{-1}(\lambda^\alpha(s,\theta)-\lambda^\alpha(s,\theta^{*}))}1_{\{\lambda^\alpha(s,\theta^{*}) \neq 0\}}ds} \to^{\proba} 0.$$

	For the middle term, consider the process 
	
		$$U_T^\alpha(\theta) = \frac{1}{T}\int_0^T{(\partial_\theta \lambda^\alpha)^{\otimes 2}(s,\theta)\lambda^\alpha(s,\theta)^{-2}\lambda^\alpha(s,\theta^{*})1_{\{\lambda^\alpha(s,\theta^{*}) \neq 0\}}ds},$$
		which, evaluated at point $\theta^{*}$ equals
		$$U_T^\alpha(\theta^{*})=\frac{1}{T}\int_0^T{(\partial_\theta \lambda^\alpha)^{\otimes 2}(s,\theta^{*})\lambda^\alpha(s,\theta^{*})^{-1}1_{\{\lambda^\alpha(s,\theta^{*}) \neq 0\}}ds}.$$
		
		For $\theta \in V_T$,
		
		\beas
		|U_T^\alpha(\theta)-U_T^\alpha(\theta^{*})| &\leq& \frac{|\theta- \theta^{*}|}{T}\int_0^T{\l|\partial_\theta\l( \frac{(\partial_\theta \lambda^\alpha)^{\otimes 2}(s,\theta)}{\lambda^\alpha(s,\theta)^{2}}\r) \r|\lambda^\alpha(s,\theta^{*})ds}\\
		&\leq& \frac{|\theta- \theta^{*}|}{T}\int_0^T{2\sup_{\theta \in \Theta}\l( \frac{|\partial_\theta \lambda^\alpha(s,\theta) ||\partial_\theta^2 \lambda^\alpha(s,\theta)|}{ |\lambda^\alpha(s,\theta)|^{2}} + \frac{|\partial_\theta \lambda^\alpha(s,\theta)|^2 |\partial_\theta^2 \lambda^\alpha(s,\theta)|}{|\lambda^\alpha(s,\theta)|^{3}}\r)\lambda^\alpha(s,\theta^{*})ds} \\
		&=& O_\proba(\text{diam}[V_T]).\\
		\eeas
	
	Finally, apply Proposition \ref{extensionA3} to $U_T^\alpha(\theta^{*})$ and write $\Gamma \in \reels^{n \times n}$ the limit of $\sum_{\alpha \in \mathbf{I}}{U_T^\alpha(\theta^{*})}$ to conclude.
\end{proof}

\begin{lemma*}
	We have :
	$$ \l(\frac{1}{\sqrt{T}}\partial_\theta l_{uT}(\theta^{*})\r)_{u \in [0,1]} \to^{d} \Gamma^{\frac{1}{2}}(W_u)_{u\in [0,1]},$$
	where $W$ is a standard Brownian motion (where the convergence happens in the Skorokhod space $\mathbf{D}([0,1])$).
	\label{FTCL}
\end{lemma*}

\begin{proof}
We consider the process 
\beas
S_u^T = \sum_{\alpha \in \mathbf{I} }{\int_0^{uT}{\frac{1}{\sqrt T}\lambda^\alpha(s,\theta^{*})^{-1}\partial_\theta \lambda^\alpha(s,\theta^{*})1_{\{\lambda^\alpha(s,\theta^{*}) \neq 0\}}d\tilde{N}_s^\alpha}}, 
\eeas
and show a functional central limit theorem when $T \to \infty$.  Indeed, as $\Delta N^{\alpha_1} \Delta N^{\alpha_2} = 0 $ almost surely for $\alpha_1 \neq \alpha_2$, we have 

\beas
\langle S^T, S^T\rangle_u =u\sum_{\alpha \in \mathbf{I} }{\int_0^{uT}{\frac{1}{ uT}\lambda^\alpha(s,\theta^{*})^{-1}\partial_\theta \lambda^\alpha(s,\theta^{*})^2 1_{\{\lambda^\alpha(s,\theta^{*}) \neq 0\}}ds}},
\eeas
and by ergodicity we deduce that 

\beas
\langle S^T, S^T\rangle_u \to^\proba u\Gamma.
\eeas

We now check Lindenberg's condition. For any $a >0$, 

\beas
	\esp \sum_{s \leq u} (\Delta S_s^T)^2 1_{\{ |\Delta S_s^T| > a \}}  &\leq& \esp \inv{a} \sum_{s \leq u} |\Delta S_s^T|^3 \\
	&=& \esp \inv{a} \sum_{\alpha \in \mathbf{I}} \int_0^{uT}{\left|\frac{\lambda^\alpha(s,\theta^{*})^{-1}\partial_\theta \lambda^\alpha(s,\theta^{*})}{\sqrt T}\right|^3 dN_s^\alpha }\\
	&=&\esp \inv{a} \sum_{\alpha \in \mathbf{I}} \frac{1}{T^{\frac{3}{2}}}\int_0^{uT}{\left|\lambda^\alpha(s,\theta^{*})^{-1}\partial_\theta \lambda^\alpha(s,\theta^{*})\right|^3 \lambda^\alpha(s,\theta^{*})ds }\\
	&\leq& \sum_{\alpha \in \mathbf{I}} \frac{1}{T^{\frac{3}{2}}}\int_0^{uT}{\esp \l[\sup_{\theta \in \Theta} \left|\lambda^\alpha(s,\theta^{*})^{-1}\partial_\theta \lambda^\alpha(s,\theta^{*})\right|^3  \lambda^\alpha(s,\theta^{*})\r] ds}\\
	&=& O\l(\frac{1}{\sqrt T}\r) \to 0.\\
\eeas

The conclusion then holds applying 3.24, chapter VIII, in \cite{JacodLimit2003}.
\end{proof} 

We finally establish the asymptotic normality from the previous lemmas :

\begin{proof}[Proof of Theorem \ref{thmAsymptoticNormality}]
	Let $\bar{\theta}_T$ be the value that maximizes $\theta \to l_T(\theta)$.  Thanks to [A2], we put $\zeta_T^1 \in [\bar{\theta}_T, \hat{\theta}_T]$ and $\zeta_T^2 \in [\theta^{*}, \hat{\theta}_T]$ such that 

	\beas
	\partial_\theta l_T(\hat{\theta}_T) = \partial_\theta^2 l_T(\zeta_T^1)(\hat{\theta}_T-\bar{\theta}_T) = o_\proba \l(\sqrt T\r)
	\eeas
	on one hand, and
	\beas
	\partial_\theta l_T(\hat{\theta}_T) = \partial_\theta l_T(\theta^{*}) + \partial_\theta^2 l_T(\zeta_T^2)(\hat{\theta}_T-\theta^{*})
	\eeas
	on the other hand. This yields, after scaling by $\sqrt T$,
	\beas
     -\frac{\partial_\theta l_T(\theta^{*})}{\sqrt T} + o_\proba \l(1\r) = \frac{\partial_\theta^2 l_T(\zeta_T^2)}{T} \sqrt T (\hat{\theta}_T-\theta^{*}).
	\eeas

	We then multiply by $-\Gamma^{-1}$ on both sides and use Lemma \ref{Fisher} and Lemma \ref{FTCL}, and the asymptotic normality follows. 
\end{proof}

\subsection{The general QLA}

We slightly strengthen assumptions about ergodicity and derivability of the intensity process in order to apply the general Quasi Likelihood Analysis. In particular, this yields the convergence of moments of the QMLE as well as the convergence of moments for the QBE. From now on, we will be only interested in the exact QMLE, that is the estimator that maximizes the Quasi Log Likelihood $\theta \to l_T(\theta)$.\\


\bd
\im[[B1\!\!]] 
The mapping $\lambda : \Omega \times \reels_+ \times \Theta \to \reels_+$ is $\calf \otimes \mathbf{B}(\reels_+) \otimes \mathbf{B}(\Theta)$-measurable. Moreover, almost surely,
  \bd
		\im[{\bf (i)}] for any $\theta \in \Theta$, $s \to \lambda(s,\theta)$ is left continuous.
		\im[{\bf (ii)}] for any $s \in \reels_+$, $\theta \to \lambda(s,\theta)$ is in $C^4(\Theta)$, and admits a continuous extension to $\overline{\Theta}$.
	\ed

\im[[B2\!\!]] 
The intensity processes and their first derivatives satisfy 
	\bd
		\im[(i)] for any $p > 1$, $\sup_{t \in \reels_+}  \sum_{i=0}^4{\left\| \sup_{\theta \in \Theta}\partial_\theta^i \lambda(t,\theta)\right\|_p} <+\infty$.
		\im[(ii)] for any $p > 1$, for any $\alpha \in \mathbf{I}$, $\sup_{t \in \reels_+} \left\| \sup_{\theta \in \Theta}\l|\lambda^\alpha(t,\theta)^{-1}\r| 1_{\{\lambda^\alpha(t,\theta) \neq 0\}}\right\|_p <+\infty$.
		\im[(iii)] For any $\theta \in \Theta$, for any $\alpha \in \mathbf{I}$, $\lambda^\alpha(t,\theta) = 0 $ if and only if  $\lambda^\alpha(t,\theta^{*}) =0$.
	\ed
\ed

The ergodicity assumption needs to be strengthened as well. The key point in the following assumption is to postulate the existence of a rate of convergence for the LLN. Let $D_\uparrow(E,\reels)$ be the set of functions $\psi$ that are of class $C^1$ on $\l(\reels_+ - \{0\} \r) \times \l(\reels_+ - \{0\} \r)\times \reels^n$ and such that $\psi, \l|\nabla \psi \r| \in C_\uparrow(E,\reels) $.

\bd
\im[[B3\!\!]] 
For any $\alpha \in \mathbf{I}$, there exists a mapping $\pi_\alpha : D_\uparrow(E,\reels) \times \Theta \to \reels$ and there exists $0 < \gamma < \frac{1}{2}$ such that for any  $(\psi,\theta) \in D_\uparrow(E,\reels) \times \Theta $ and for every $p >1$ the following convergence holds :
	\beas
	\sup_{\theta \in \Theta} T^\gamma\l\|\frac{1}{T}\int_0^T{\psi(\lambda^\alpha(s,\theta^{*}),\lambda^\alpha(s,\theta),\partial_\theta \lambda^\alpha(s,\theta))ds} - \pi_\alpha(\psi,\theta) \r\|_p \to 0. 
	\eeas

\im[[B4\!\!]]
Define $\chi_0$ as: 
$$\chi_0 = \inf_{\theta \in \Theta \backslash\{\theta^{*}\}} -\frac{\mathbb{Y}(\theta)}{|\theta-\theta^{*}|^2}.$$

We assume that 
$$\chi_0 >0.$$ 

\ed

For a given prior density $p$ on the space $\Theta$, we define the Quasi-Bayesian Estimator as follows: 

\beas
\tilde{\theta}_T = \l[ \int_\Theta{\text{exp}(l_T(\theta))p(\theta)d\theta}\r]^{-1}\int_\Theta{\theta \text{exp}(l_T(\theta))p(\theta)d\theta}.
\eeas

We are going to apply results from \cite{YoshidaPolynomial2011} about Quasi Likelihood Analysis based on polynomial type large deviations. This approach follows techniques that have been developed in \cite{ibragimov2013statistical} to prove the asymptotic properties of the MLE and the BE in a general context. In particular, the convergence of moments of those estimators is obtained as a consequence of the convergence in distribution of the Likelihood random field along with a polynomial type large deviation inequality similar to the one stated in Theorem \ref{thmQLA}. In our case, we adopt the notations of \cite{YoshidaPolynomial2011} and represent the Quasi Likelihood process as follows. For any $u \in U_T = \{ u \in \reels^n | \theta + T^{-1/2}u \in \Theta \}$, we write $\theta_u = \theta + T^{-1/2}u$ and define the Quasi Likelihood random field 

\beas
Z_T(u) = \text{exp}\{l_T(\theta_u) - l_T(\theta^{*})\}.
\eeas 

The next theorem gives a large deviation inequality on $Z_T$, which in turn is widely used to establish the convergence of moments of the estimators.    
\begin{theorem*}
	Under \textnormal{[B1]-[B4]}, the two following results hold: 
	
	\bd
		\im[Polynomial type large deviation inequality] 
		for every $L>0$, there exists $C_L$ such that :
	$$\proba \left[ \sup_{u \in U_T, |u| >r} Z_T(u) \geq e^{-r}\right] \leq \frac{C_L}{r^L}. $$
		\im[Convergence of moments] 
		If $\hat{\theta}_T$ is the QMLE and $\tilde{\theta}_T$ the QBE, we have :
	$$\esp \left[ f(\sqrt{T}(\hat{\theta}_T-\theta^{*})) \right] \to \esp[f(\Gamma^{-\frac{1}{2}}\xi)],$$
	$$\esp \left[ f(\sqrt{T}(\tilde{\theta}_T-\theta^{*})) \right] \to \esp[f(\Gamma^{-\frac{1}{2}}\xi)],$$ 
	for any continuous $f$ of polynomial growth, and such that $\xi$ follows a standard normal distribution.
	\ed
	\label{thmQLA}
\end{theorem*}

To prove this result, and in order to highlight the LAN property (at point $\theta^{*}$) of the model, we rewrite $Z_T(u)$ as 

\bea
 Z_T(u) = \text{exp}\l\{\Delta_T[u] -\frac{1}{2} \Gamma_T[u,u] + r_T[u]\r\},
\eea
with 

\beas
\Delta_T[u] &=& \frac{u}{\sqrt{T}} \partial_\theta l_T(\theta^{*}) 
\eeas
and 

\beas
\Gamma_T[u,u] = -\partial_{\theta}^2 l_T(\theta^{*})\l[\frac{u^{\otimes 2}}{T}\r].
\eeas

Finally $r_T[u]$ is defined as the following residual in the Taylor formula :
\beas
 l_T(\theta_u)=l_T(\theta^{*})+\Delta_T[u] -\frac{1}{2} \Gamma_T[u,u]  + r_T[u].  
\eeas

The following technical lemma is proved in the appendix :

\begin{lemma*}
\label{lemmaQLA}
Under \textnormal{[B1]-[B3]}, for every $p>1$,
\bea
\sup_{T \in \reels_+} \left \| \Delta_T \right \|_p < \infty,
\label{eqDelta}
\eea

\bea
\sup_{T \in \reels_+} \left\| T^\gamma \sup_{\theta \in \Theta} |\mathbb{Y}_T(\theta)-\mathbb{Y}(\theta)| \right\|_p < \infty, 
\label{eqY}
\eea

\bea
\sup_{T \in \reels_+} \left\| T^\gamma |\Gamma_T - \Gamma| \right\|_p < \infty, 
\label{eqGamma}
\eea

\bea
\sup_{T \in \reels_+} \left\| T^{-1} \sup_{\theta \in \Theta}|\partial_\theta^3l_T(\theta)| \right\|_p < \infty. 
\label{eqPartial}
\eea

\end{lemma*}

Let us deal with the proof of Theorem \ref{thmQLA}.

\begin{proof}[Proof of Theorem \ref{thmQLA}]
We first show the polynomial type large deviation inequality. We apply Theorem 3 in \cite{YoshidaPolynomial2011}. Setting $\beta_1 = \gamma$, $\beta_2=\frac{1}{2} - \gamma$, $\rho =2$, $\rho_2 \in ] 0, 2\gamma [$, $\alpha \in ]0,\frac{\rho_2}{2}[$, and $\rho_1 \in ]0,\min \{1,\frac{\alpha}{1-\alpha},\frac{2\gamma}{1-\alpha} \} [$, Conditions $(A1^{''})$, $(A4^{'})$, $(A6)$ in \cite{YoshidaPolynomial2011} are satisfied thanks to Lemma \ref{lemmaQLA}, as well as $(B1)$, $(B2)$ thanks to the non-degeneracy assumption [B4] above.\\
\medskip

We now make use of Theorem 4 in \cite{YoshidaPolynomial2011} to show the convergence of moments for the QMLE. We first extend the definition of $\mathbb{Z}_T(u)$ to any $u \in \reels^n$ by taking $\mathbb{Z}_T(u) $  continuously decreasing to zero outside $U_T$. We need to show that the finite dimensional distribution of $\mathbb{Z}_T$ are convergent, and then that $\text{log}\mathbb{Z}_T$ is tight in $T$, seen as a family of processes in $u \in K $ for any compact set $K$ of $\reels^n$. Because we have the majoration  

$$ \esp \l[|r_T[u]|\r] \leq T^{-\frac{3}{2}}\esp \l[ \sup_{\theta \in \Theta} |\partial_\theta^3 l_T(\theta) ||u|^3 \r] = O\l( \inv{\sqrt T}\r)\to 0,$$ 
the finite dimensional convergence, and thus the LAN property, is a direct consequence of the classical approach given the expressions of $\Delta_T$ and $\Gamma_T$.\\

\medskip

Take then an arbitrary compact set $K$, and put $w_T(\delta) = \sup_{|u_2-u_1| \leq \delta} |\text{log} \mathbb{Z}_T(u_2)-\text{log} \mathbb{Z}_T(u_1)|$ where the supremum is taken over the set $K$. We need to prove that for every $\epsilon >0$,

\beas
 \lim_{\delta \to 0} \sup_T \proba [w_T(\delta) \geq \epsilon] =0.
 \eeas

Using Markov's inequality we first have for amy $p>n$
\beas
\proba [w_T(\delta) \geq \epsilon]&\leq& \epsilon^{-p}\esp|w_T(\delta)|^p\\
&\leq& \epsilon^{-p}\esp \sup_{|u_2-u_1| \leq \delta} \l|l_T(\theta_{u_2})-l_T(\theta_{u_1}) \r|^p.
\eeas

We have
$$l_T(\theta_{u_2}) - l_T(\theta_{u_1})=\sum_{\alpha \in \mathbf{I}}\int_0^T{\text{log}\frac{\lambda^\alpha(s,\theta_{u_2})}{\lambda^\alpha(s,\theta_{u_1})}d\tilde{N}_s^\alpha}-\int_0^T{\l\{\text{log}\frac{\lambda^\alpha(s,\theta_{u_2})}{\lambda^\alpha(s,\theta_{u_1})}-\frac{\lambda^\alpha(s,\theta_{u_2})-\lambda^\alpha(s,\theta_{u_1})}{\lambda^\alpha(s,\theta^{*})}\r\}\lambda^\alpha(s,\theta^{*})ds},$$
that we bound from above in two steps. Defining $M_T^\alpha(\theta_u) = \int_0^T{\text{log}\lambda^\alpha(s,\theta_u)d\tilde{N}_s^\alpha}$, we first write for some $\alpha \in \mathbf{I}$ and $p > n$,

\beas
\esp \sup_{|u_2-u_1| \leq \delta}  \l|\int_0^T{\text{log}\frac{\lambda^\alpha(s,\theta_{u_2})}{\lambda^\alpha(s,\theta_{u_1})}d\tilde{N}_s^\alpha}\r|^p &=& \esp \sup_{|u_2-u_1| \leq \delta}  \l| M_T^\alpha(\theta_{u_2})-M_T^\alpha(\theta_{u_1})\r|^p\\
&\leq& \delta^p T^{-\frac{p}{2}}\esp \sup_{u \in K} \l| \partial_\theta M_T^\alpha(\theta_u) \r|^p\\
&\leq& \K_1 \delta^p
\eeas
where we have applied Sobolev's inequality and Davis-Burkholder-Gundy's inequality at the last step. One can check that the same holds for the integral with respect to Lebesgue measure:

\beas
\esp \l|\sup_{|u_2-u_1| \leq \delta} \int_0^T{\l\{\text{log}\frac{\lambda^\alpha(s,\theta_{u_2})}{\lambda^\alpha(s,\theta_{u_1})}-\frac{\lambda^\alpha(s,\theta_{u_2})-\lambda^\alpha(s,\theta_{u_1})}{\lambda^\alpha(s,\theta^{*})}\r\}\lambda^\alpha(s,\theta^{*})ds} \r|^p &\leq& \K_2 \delta^{p}.
\eeas

 This shows that 

\beas
\sup_{T \in \reels_+} \proba [w_T(\delta) \geq \epsilon]&\leq&  \frac{\K \delta^p}{\epsilon^p} \to 0,\text{  } \delta \to 0.
\eeas

Finally, thanks to the polynomial type large deviation inequality and Lemma 2 in \cite{YoshidaPolynomial2011}, for some $\delta >0$ we have 
\beas
\sup_{T \in \reels_+}\esp \l[ \l(\int_{u:|u| \leq \delta}{Z_T(u)du}\r)^{-1}\r] < \infty.
\eeas

The convergence of moments of the QBE is then a direct consequence of Theorem 8 in \cite{YoshidaPolynomial2011}.
\end{proof}

\subsection{Mixing criteria for ergodicity}

We conclude this theoretical part by giving two mixing criteria that respectively imply the ergodicity conditions [A3] and [B3]. As we shall see in the sequel, it is sometimes easier to check such condition, as is the case for
the Hawkes processes. In order to match with our definition \ref{defErgo}, we say that a process $(X_t)_{t \in \reels_+}$ taking values in some state space $E$ is $C$-mixing, for some set of functions $C$ from $E$ to $\reels$, if for any $\phi$, $\psi \in C$, the following convergence holds 

$$ \rho_{u}= \sup_{ s \in \reels_+} \text{Cov}[\phi(X_s),\psi(X_{s+u})] \to 0,\text{  } |u| \to +\infty $$

\bd
\im[[M1\!\!]] The following two properties hold :
\bd
\im  [Mixing] For any $\alpha \in \mathbf{I}$, $(\lambda^\alpha(t,\theta^{*}),\lambda^\alpha(t,\theta),\partial_\theta \lambda^\alpha(t,\theta))_{t \in \reels_+}$ is $C_b(E,\reels)$-mixing.  
\im  [Stability] For any $\alpha \in \mathbf{I}$, there exists $\bar{\lambda}^\alpha$ such that for any $\theta \in \Theta$,
  $$(\lambda^\alpha(t,\theta^{*}),\lambda^\alpha(t,\theta),\partial_\theta \lambda^\alpha(t,\theta)) \to^d (\bar{\lambda}^\alpha(\theta^{*}),\bar{\lambda}^\alpha(\theta),\partial_\theta \bar{\lambda}^\alpha(\theta)).$$
\ed

\ed

Note that in the stationary case with $C = \{ 1_A | A \in  \mathbf{B}(\reels) \}$, the above condition is nothing more than a mixing assumption in the classical sense (see e.g. \cite{billingsley1965ergodic}) which is a well-known sufficient condition for the ergodicity of the invariant measure. In the same spirit of [B3], we also reformulate [M1] with a minimal rate of convergence in the mixing and the stability equations. 

\bd
\im[[M2\!\!]] There exists $0 < \gamma <\half$ such that :
\bd
\im [Mixing] For any $\alpha \in \mathbf{I}$, $(\lambda^\alpha(t,\theta^{*}),\lambda^\alpha(t,\theta),\partial_\theta \lambda^\alpha(t,\theta))_{t \in \reels_+}$ is $D_\uparrow(E,\reels)$-mixing uniformly in $\theta \in \Theta$. Moreover the rate $\rho$ verifies :
$$\rho_u = o(u^{-\epsilon})  \text{ for some } \epsilon > \frac{2\gamma}{1-2\gamma}.$$ 
\im [Stability] There exists $\bar{\lambda}^\alpha$ such that for any $\psi \in D_\uparrow(E,\reels) $,
  $$\sup_{\theta \in \Theta} t^\gamma\l|\esp[\psi(\lambda^\alpha(t,\theta^{*}),\lambda^\alpha(t,\theta),\partial_\theta \lambda^\alpha(t,\theta))] - \esp[\psi(\bar{\lambda}^\alpha(\theta^{*}),\bar{\lambda}^\alpha(\theta),\partial_\theta \bar{\lambda}^\alpha(\theta))]\r| \to 0.$$
\ed

\ed

The next lemma links the above mixing criteria to the ergodicity assumptions [A3] and [B3] and specifies the nature of the operator $\pi$ in those cases.

\begin{lemma*}
If \textnormal{[M1]}, \textnormal{[A1]} and \textnormal{[A2]}  (resp. \textnormal{[M2]}, \textnormal{[B1]} and \textnormal{[B2]}) are satisfied, then so is the ergodicity condition \textnormal{[A3]} (resp. \textnormal{[B3]}), and moreover the mapping $\pi_\alpha$ writes:

$$\pi_\alpha(\psi, \theta) = \esp [ \psi(\bar{\lambda}^\alpha(\theta^{*}),\bar{\lambda}^\alpha(\theta),\partial_\theta \bar{\lambda}^\alpha(\theta))].$$
\label{lemmaM1}
\end{lemma*}


\begin{table}[h!]

 \centering
  \begin{tabular}{| l | c | c | c |}
   \hline 
    Estimator & Asymptotic QMLE & Exact QMLE & QBE \\
   \hline
    Definition & $l_T(\theta_T) \geq \sup_{\theta \in \Theta} l_T(\theta) - o_\proba(\sqrt{T})$   &$l_T(\theta_T) = \sup_{\theta \in \Theta} l_T(\theta)$  & $\theta_T = \frac{\int_\Theta{\theta \text{exp}(l_T(\theta))p(\theta)d\theta}}{ \int_\Theta{\text{exp}(l_T(\theta))p(\theta)d\theta}}$ \\
   \hline 
    Regularity Assumptions & [A1], [A2] and [A4]& \multicolumn{2}{c|}{[B1], [B2] and [B4]} \\
   \hline
    Ergodicity Assumption & [A3] or [M1]&\multicolumn{2}{c|}{[B3] or [M2]}\\
    \hline
     \multirow{2}{12mm}{Property} &$\esp[f(\sqrt{T}(\theta_T-\theta^*))]\to \esp[f(\Gamma^{-\half}\xi)]$&\multicolumn{2}{c|}{$\esp[f(\sqrt{T}(\theta_T-\theta^*))]\to \esp[f(\Gamma^{-\half}\xi)]$}\\
     &$\forall f \in C_b(\reels^n,\reels)$   & \multicolumn{2}{c|}{$\forall f \in C_p(\reels^n,\reels)$}\\
   \hline
  \end{tabular}
  \caption{Summary of the asymptotic properties of the three estimators presented in Section 3. In each definition, $\theta_T$ represents the estimator itself. $C_b(\reels^n,\reels)$ (resp. $C_p(\reels^n,\reels)$) is the set of functions from $\reels^n$ to $\reels$ which are continuous and bounded (resp. continuous and of polynomial growth) .  }

\label{tableSummary}
 \end{table}

\section{Applications}

As mentioned in the introduction, a fruitful approach to Limit Order Book process modelling consists in focusing on the point process that counts different events of interest occurring over time. More precisely, recall that in a simplified representation of the continuous-time double auction system any agent interacts with the market through three types of orders :
\begin{itemize}
\im Limit order : Submit a buy (resp. sell) order at a lower (resp. higher) price than the best ask (resp. bid) price. The order is immediately inserted in the corresponding queue (regarding the price level $p$ at which the order has been sent).
\im Market order : Consume the liquidity at the best ask (resp. bid) price for a buy (resp. sell) order.
\im Cancellation order : Remove one limit order that is waiting in the LOB.
\end{itemize}

 In practice every order is characterized by a volume $v$ that represents the amount of shares one is willing to buy or sell, and a price level $p$ that is constrained to belong to a \textit{price grid} $\calg \subset \{k.\Delta p | k \in \naturels \}$. $\Delta p$ is the smallest distance between two price limits and is called the \textit{ticksize}. Let us then call a (finite dimensional) Limit Order Book a process that is constituted of $m \in \naturels$ queues stored in $\mathbb{X}(t) \in \relatifs^{m}$. At time $t \in \reels_+$, $\mathbb{X}_t^\alpha$ contains the total volume of limit orders that have been submitted through time at a given level and that have not been executed yet. As sellers should always submit limit orders at higher prices than buyers, A Limit Order Book can always be split into two distinct parts : The queues associated to selling limit orders, at higher prices, or the \textit{ask side}, and the queues at lower prices, that contain buying limit orders, or the \textit{bid side}. It is also common to call \textit{spread} the empty zone between the best limits, that is the limits laying between the highest bid and the lowest ask. \\
\medskip

 As the number of queues should be potentially infinite, such modelling is not straightforward, and many representations exist in the literature. Most of those representations can be actually classified into two main families : The \textit{absolute} LOB's associate to each price level $p \in \calg$ a label $\alpha \in \naturels$ and for the whole life of the process $\mathbb{X}$, $\mathbb{X}_\alpha$ will contain the volume of latent limit orders that have been submitted at this price level. It is common to count positively the limit orders on the \textit{ask side}, and negatively those on the \textit{bid side}. Such representations only involve a finite price range $\{k_1 \Delta_p,..., k_2 \Delta_p\}$, and are thus satisfactory provided that most events occur in this window. An illustration of this model can be found in \cite{ContStoikovTalrejaStochastic2010}. On the other hand, the \textit{relative} LOB's are generally centered around a \textit{reference price}, that is a price that either plays a role for the shape of the LOB itself, or has a good economical interpretation. As an example, the former is often taken as the best ask price or the best bid price, see \cite{AbergelLOB2013}. The latter can be an underlying unobserved process that is assumed to be the \textit{true price}, that is the value toward which the market price would tend if, say, the information was perfectly shared. In any case, the relative representation associates a label $\alpha$ to each price level relatively to its position compared to the \textit{reference price}. The general framework introduced in \cite{huang2015ergodicity} is based on this representation. Therefore, as such price is typically stochastic, any price movement shifts the labelling to the left or to the right as time passes. Both representations have their pros and cons, that we will not discuss here. In the following we are not concerned with this technical choice since the statistical procedure remains the same. \\

\medskip

 Since any order is characterized by a (random) volume $v$, the jumps of $\mathbb{X}$ need not be of size $1$. To conduct the QLA through a point process perspective, it is quite common to simply decompose $\mathbb{X}$ through the arrival of orders of different types, and ignore in a first approximation the size of those orders. We thus define   

\begin{itemize}
\im An $m$-dimensional point process $L$ that counts the limit orders, compensated by $\Lambda^L$.
\im An $m$-dimensional point process $C$ that counts the cancellation orders, compensated by $\Lambda^C$.
\im A $2$-dimensional point process $M$ that counts the market orders, compensated by $\Lambda^M$.
\end{itemize}

Note that it is possible to take into account the size of the orders by splitting again every point process according to the volume of the orders. A general approach would therefore consist in constructing a point process that counts the orders of type $Y \in \{C,L,M\}$, at price level $\alpha \in \{1,...,m\}$, and of volume $v \in \{1,...,\bar{v}\}$ for any triplet $(Y,\alpha,v)$. Such decomposition is done in a Markovian framework in \cite{huang2015ergodicity}. In practice, $\bar{v}$ is large and it is thus preferable for dimension considerations to construct $K$ zones for the volume $V_1,...,V_K$ and collect the orders whose volume belongs to a given zone. Other criteria are sometimes taken into account when considering the modelling of an Order Book. In \cite{toke2015stationary}, the author shows that aggressiveness of orders is also a very important feature when it comes to describing the stationary distribution of best limits sizes. Basically, an order is said to be aggressive if it is a market order that consumes all the liquidity at one of the best limits, or if it is a limit order placed inside the spread.\\
\medskip

It is also worth noting that the knowledge of $\mathbb{X}$ is sufficient to reconstruct the process $L$ since it counts the jumps that increase the absolute value of the size of one queue. On the other hand, market orders and cancellation orders at best limits cannot be distinguished from the knowledge of $\mathbb{X}$ alone. Practically speaking, when working with real data, it is crucial to assume that the events that diminish the size of one queue are labelled in one way or the other so one can distinguish market and cancellation orders. This is possible if for example one has access to the record of trades on the market. Other practical issues, such as hidden liquidity, or cross trades can also make the data analysis difficult. \\
\medskip

Let us assume from now on that the LOB $\mathbb{X}$ and the point process $N$ that gathers the events of interest are constructible, and fully observable. We focus on the parametrization of the intensities along with the ergodicity of the models existing in the literature.  

\subsection{Markovian models}

Some models assume the existence of an observable underlying process $(Y_t)_t$ adapted to the filtration $\F=(\calf_t)_{t\in\bbR_+}$ such that the parametrized intensity process writes 

\beas
	\lambda(t, \theta) &=& h(Y_t,\theta). 
\eeas

\begin{example*}
Consider first the very simple following model, inspired of \cite{ContLarrardPrice2013}. The LOB $(\mathbb{X}(t))_t$ is represented in a level-I perspective (best ask and best bid limits only). It is assumed that every time one queue gets empty, it is randomly regenerated. This accounts for the volume that is stored in the new best limit, that is the first non-empty limit after the one that just depleted, and that is inaccessible in this two dimensional representation. The authors postulate that all the intensities are constant, and thus $(L_t,M_t,C_t)_t$ is a $6$-dimensional Poisson process :

\beas
\lambda(t) &=& \lambda_0.
\eeas 

The multivariate Poisson process being obviously ergodic, this toy model trivially enjoys all the properties derived in Theorem \ref{thmQLA}.  
\end{example*}

\begin{example*}
In \cite{AbergelLOB2013}, Abergel and Jedidi define a multidimensional LOB whose cancellation intensities are linear functions of the size of the queues of the LOB itself, and other intensities remain constant. In other words, any intensity at a given level $\alpha \in \mathbf{I}$ is of the form 

\beas
\lambda^{L,\alpha}(t) &=& \lambda_0^{L,\alpha},\\
\lambda^{C,\alpha}(t) &=& \lambda_0^{C,\alpha} |\mathbb{X}_\alpha(t-)|, 
\eeas
and the market orders intensities on the best limits write 
\beas
\lambda^{M,\text{Bid}}(t) =  \lambda_0^{M,\text{Bid}},\\
\lambda^{M,\text{Ask}}(t) = \lambda_0^{M,\text{Ask}}.
\eeas

Such model presupposes that market agents are independent, and that every order that has been posted is cancelled after exponential random times. In their paper, Abergel and Jedidi prove that such feedback mechanism along with the condition $\lambda_0^{C,\alpha} >0$ for all $\alpha \in \mathbf{I}$ are sufficient to show that  $(\mathbb{X}_t)_t$ is a $V$-geometric Markovian process (see \cite{MeynTweedieMarkovChain2009,MeynTweedieprocessiii1993} for a deep insight of this notion).  
\label{exampleAbergel1}
\end{example*}

\medskip

\medskip
The above mentioned models and their variants are reputed for their mathematical tractability. For example, analytical expressions of quantities of interest that quantify the shape of the Limit Order Book are derived in \cite{MuniTokeQueuing2013}. Needless to say that being able to write closed formulas is generally admitted to be a significant criterion for practical considerations. On the other hand, it is also equally admitted that such Poisson process-based models perform poorly when it comes to closely reproducing most stylized facts observed on the markets, as it is shown in \cite{RosenbaumQueue2014}. In particular, the lack of dependency between the successive orders, and between the queues is one of the major drawbacks of those basic models. The following example is a quite general Markovian framework that allows such dependencies.

\begin{example*}
In \cite{RosenbaumQueue2014} and \cite{huang2015ergodicity}, the model described in Example \ref{exampleAbergel1} is generalized. $\mathbb{X}$ is defined as a general Pure Jump-Type Markovian process, as in Example \ref{purejumpex}. Recall that the intensity vector is therefore a pure function of the LOB state :
\beas
\lambda(t,\theta) &=& h(\mathbb{X}(t-),\theta).
\eeas

Whatever the form of $h$, the authors establish the $V$-geometric ergodicity of the Markovian process $\mathbb{X}$  in the case where cancellation intensities become predominant whenever the size of one limit gets too large. 
\end{example*}


\subsection{Multivariate Hawkes process}

In this section we are interested in the special case of multivariate Hawkes process with exponential kernel as Hawkes process-based models have been applied to various topics in finance, notably to Limit Order Books. As we shall see in the remaining part of this paper, it turns out that the full QLA is applicable to such class of models. Both computation of the MLE and simulation methods for the exponential Hawkes process can be found in \cite{ozaki1979maximum}.\\
\medskip

Those processes were introduced by Hawkes in 1971, see \cite{HawkesPointSpectra1971} and were extensively used to model earthquakes and their aftershocks. Lately they have been used in finance to model various phenomena such as price variation, market impact, and Limit Order Book mechanisms. \cite{BacryMastromatteo2015} provides a very complete summary of the different applications of Hawkes processes in finance. Let us mention in particular the use of Hawkes processes in \cite{muniToke2010} to model the interaction between liquidity providers and liquidity takers. In this work a level-I LOB is described, in which the point process gathering limit and market orders is modelled as a multivariate Hawkes process. The use of those self-exciting processes highlights the fact that orders seem to trigger each other. It quantifies to what extent self excitations and mutual excitations have an impact on the interarrival time of order submissions, and their relative importance. For example, it is shown in \cite{muniToke2010} that market orders (and thus liquidity takers) strongly influence limit orders (liquidity providers), but limit orders, in turn, have a negligible impact on market orders. Recently, a more theoretical work by Abergel and Jedidi \cite{AbergelHawkes2015} establishes $V$-geometric ergodicity and scaling limits in a Hawkes-based Markovian framework applied to a general LOB. It is basically a generalization of Example \ref{exampleAbergel1}, where the Poisson assumption on limit orders and market orders is relaxed. More specifically, keeping the same notations as above, instead of a mere Poisson process, the vector $(L,M)$ is supposed to form a multivariate Hawkes process while the dynamic of $C$ remains unchanged. Motivated by those examples whose detailed descriptions can be found in the abovementioned papers, we turn to the definition of a general Hawkes process. \\

\medskip
 Let $N_t = (N_t^\alpha)_{\alpha \in \mathbf{I}}$, $\mathbf{I} = \{1,...,d\}$, $N_0 = 0$, be a multidimensional point process and write $\F^N=(\calf_t^N)_{t\in\bbR_+}$, where we recall that $\calf_t^N = \sigma\{N_s | 0 \leq s \leq t \}$ is the canonical filtration of $N$. We say that $N$ is a linear Hawkes process or Hawkes' self-exciting process starting from $0$ if there exist $h : \reels_+ \to \reels_+^{d \times d}$ and $\nu \in (\reels_+^{*})^d$  such that the $\calf_t^N$-intensity $\lambda(t)$ of $N$ writes 

\bea
\lambda^\alpha(t) = \nu_\alpha +\sum_{\beta \in \mathbf{I}}\int_0^{t-}{h_{\alpha\beta}(t-s)dN_s^\beta} \text{, } \alpha \in \mathbf{I}.
\eea
 The baseline intensities $\nu_\alpha$'s represent the rate of spontaneous occurences of events, while the kernels $h_{\alpha\beta}$'s model self-interaction in the system. Indeed, if a shock occurs at time $t_0$ on the covariate $N^\beta$, an aftershock will happen on the covariate $N^\alpha$ around time $t_1$ with high probability if $h_{\alpha\beta}(t_1-t_0)$ is large. When $h_{\alpha\beta} = 0$, the covariate $N^\beta$ has no influence on the chain of events related to $N^\alpha$.  For example, in the model from \cite{AbergelHawkes2015}, and in accordance with the empirical results found in \cite{muniToke2010}, the mutual excitation structure of the process $(L,M)$ is assumed to be of the form ``MM-LL-LM'', i.e market orders excite both market orders and limit orders, whereas limit orders only trigger themselves. \\
\medskip

 Let us define the matrix $\Phi = [\phi_{\alpha\beta}]_{\alpha\beta}$ where
\beas
\phi_{\alpha\beta} = \int_0^{+\infty}{h_{\alpha\beta}(s)ds},
\eeas
and write $\rho(\Phi)$ its spectral radius. We also define the elementary excitations as 

$$\epsilon_{\alpha\beta}(t) =\int_0^{t-}{h_{\alpha\beta}(t-s)dN_s^\beta},$$
that we gather in the matrix $\cale(t) = \l[ \epsilon_{\alpha\beta(t)}\r]_{\alpha\beta}$. Finally, given $A = [a_{\alpha\beta}]_{\alpha\beta}$ and $C =[c_{\alpha\beta}]_{\alpha\beta} \in \reels_+^{d \times d}$, we say that $N$ is an exponential Hawkes process if the kernel functions $h_{\alpha\beta}$ are of the form 

$$h_{\alpha\beta}(s) = c_{\alpha\beta}e^{-a_{\alpha\beta}s}. $$

 Note that the  matrix $\Phi$ has the representation $\Phi = \l[ \frac{c_{\alpha\beta}}{a_{\alpha\beta}}\r]_{\alpha\beta} $ in that case. We now recall a few results about the Hawkes process and its exponential form.

\begin{proposition*}

Let $N$ be a multivariate Hawkes process, and assume that $\rho(\Phi) < 1$. Then :
\bd
\im [(i)] There exists a probability space $(\Omega^{'}, \calf^{'}, \proba^{'})$ and two point processes $N^{'}$ and $\bar{N}$ defined on $\Omega^{'}$ and on the whole real line such that $N^{'}_{|\reels_+}$ is distributed as $N$, and $\bar{N}$ is a stationary point process whose $\calf_t^{\bar{N}}$-intensity $\bar{\lambda}$ verifies :
$$ \bar{\lambda}^\alpha(t) =  \nu_\alpha +\sum_{\beta \in \mathbf{I}}\int_{-\infty}^{t-}{h_{\alpha\beta}(t-s)d\bar{N}_s^\beta}.$$
\im [(ii)] Let $\mathbf{S}$ be the shift operator, meaning that for any $t \in \reels_+$, $\mathbf{S}_tN = (N_{s+t})_{s \in \reels_+}$. Then $N$ is stable in the following sense :

$$\mathbf{S}_t N \to^{\mathbf D} \bar{N}_{| \reels_+},  $$
where $\mathbf{D}$ designates the weak convergence associated to the vague topology on integer valued measures. 

\im [(iii)] If $N$ is an exponential Hawkes process, there exist $q > 0$ and $A > 0$ such that, for any $\alpha \in \mathbf{I}$, for any $t \in \reels_+$,

 $$ \l\| \lambda^\alpha(t) - \bar{\lambda}^\alpha(t)\r\|_1 \leq Ae^{-qt}.$$
\ed

\label{stability}
\end{proposition*}

See \cite{BremaudStability1996} for deeper explanations about this mode of convergence. The first two points of Proposition \ref{stability} are immediate consequences of Theorem $1$ and Lemma $4$ in \cite{BremaudStability1996} along with the assumption on $\rho(\Phi)$. Note that they do not require $N$ to be exponential. In practice, the first point allows us to assume the existence of a stationary version of $N$ on the same probability space, and the second assertion states that $N$ tends in distribution to $\bar{N}$ for a certain topology. Finally, the last point states that in the exponential case, the $\mathbb{L}^1$ deviation of the transient intensity from its stationary version is, not surprisingly, exponentially decreasing.

\medskip

As is well known, The process $\cale$ is Markovian if the kernel are of exponential form. It also enjoys a very strong ergodicity property, that is introduced in the next proposition. Following the definitions in \cite{MeynTweedieprocessiii1993}, we say that a continuous time Markovian Process $X$ is $V$-geometrically ergodic if, denoting by $P$ its transition kernel, there exists some norm-like function $V$ on the state space of $X$, $B < +\infty$, and $r < 1$ such that :
		\bea
		\label{Vergo}
		\l\|P^t(x,.) - P^{\bar{X}}\r\|_V &\leq& B(V(x)+1)r^t,
		\eea
		for any initial state $x$, and any $t \in \reels_+$, where $P^{\bar{X}}$ is the stationary law of $\bar{X}_t$. Here, $\l\|.\r\|_V$ designates the $V$-variation norm (See \cite{MeynTweedieMarkovChain2009}), that is, for any measure $\mu$ on a measured space $(S, \cals)$, 
		\bea
		\|\mu\|_V = \sup_{\psi | \psi \leq V} \l| \int_S \psi(s)\mu(ds) \r|.
		\label{Vnorm}
		\eea

		
		

\begin{proposition*}
\label{propVgeo}
Assume that $\rho(\Phi) < 1$. Then :
\bd
\im [(i)] The elementary excitation process $\cale$ is a Markovian process taking values in $\reels_+^{d \times d}$.
\im [(ii)] $\cale$ is $V$-geometrically ergodic, and moreover $V : \reels_+^{d \times d} \to \reels_+$ can be chosen as $V(\epsilon) = e^{\langle M, \epsilon \rangle}$ for some $M = [m_{\alpha\beta}]_{\alpha\beta} \in \reels_+^{d \times d}$ and with $\langle M, \epsilon \rangle = \sum_{\alpha,\beta \in \mathbf{I}}{m_{\alpha\beta} \epsilon_{\alpha \beta}}$.    
\ed
\end{proposition*}

By virtue of the Markovian property and the expression (\ref{Vnorm}), (\ref{Vergo}) thus writes for any $0 \leq s \leq t$,

\bea
\sup_{\psi | \psi \leq V}\l|\esp \l[ \psi(\cale_t)| \calf_{s}^{\cale}\r] - \esp \l[\psi(\bar{\cale})\r]\r|\leq B(V(\cale_s)+1)r^{t-s} 
\label{eqVgeoEpsilon}
\eea
where $\bar{\cale}$ follows the stationary distribution of the process $\cale$.\\

\medskip

We now consider an exponential Hawkes process as a model parametrized by the triplet $(\nu, C, A) \in \reels_+^d \times \reels_+^{d\times d} \times \reels_+^{d\times d}$. More precisely we consider a relatively compact open parameter state $\Theta \subset  \reels_+^d \times \reels_+^{d\times d} \times \reels_+^{d\times d}$ such that for any triplet $ \theta = (\nu, C, A) \in \Theta$, for any $\alpha, \beta \in \mathbf{I}$,

\beas
0 &<& \underline{\nu} \leq \nu_\alpha \leq \bar{\nu} < +\infty, \\
0 &<& \underline{c} \leq c_{\alpha\beta}  \leq \bar{c} < +\infty,\\
0 &<& \underline{a} \leq a_{\alpha\beta}  \leq \bar{a} < +\infty.\\
\eeas

In particular, we assume for the sake of identifiability that the $c_{\alpha\beta}$ are positive. Let us assume, in the spirit of the first part, that there exists $\theta^{*}=(\nu^{*},C^{*},A^{*}) \in \Theta$, such that $\theta^{*}$ is the real parameter that drives the dynamic of $N$. We are now ready to state the main result of this section.
\begin{theorem*}
The exponential Hawkes model verifies the conditions \textnormal{[B1], [B2], [M2],} and \textnormal{[B4]}. In particular, by virtue of Theorem \ref{thmQLA} and Lemma \ref{lemmaM1}, if $\hat{\theta}_T = (\hat{\nu}_T, \hat{C}_T, \hat{A}_T)$ is the QMLE and $\tilde{\theta}_T = (\tilde{\nu}_T, \tilde{C}_T, \tilde{A}_T)$ the QBE of the model, we have
	$$\esp \left[ f(\sqrt{T}(\hat{\theta}_T-\theta^{*})) \right] \to \esp[f(\Gamma^{-\frac{1}{2}}\xi)],$$
	$$\esp \left[ f(\sqrt{T}(\tilde{\theta}_T-\theta^{*})) \right] \to \esp[f(\Gamma^{-\frac{1}{2}}\xi)],$$ 
	where $\Gamma$ is the Fisher information, $f$ is any continuous function of polynomial growth, and $\xi$ follows a standard normal distribution.
\label{thmHawkesQLA}
\end{theorem*}

[B1] is immediate given the explicit expression of the intensity function. Unfortunately, the ergodicity condition [B3] is not a direct consequence of the $V$-geometric ergodicity of the process $\cale$. This is due to the fact that the vector $(\lambda(t,\theta^{*}),\lambda(t,\theta),\partial_\theta \lambda(t,\theta))$ cannot be expressed as a pure function of the marginal $\cale(t)$. On the other hand, since the stochastic structure of the process $(\lambda(.,\theta^{*}),\lambda(.,\theta),\partial_\theta \lambda(.,\theta))$ is globally driven by $\cale$, the exponential rates in (\ref{eqVgeoEpsilon}) and in Proposition \ref{stability} (iii) are strong enough to ensure the mixing assumption [M2]. The detailed proofs can be found in the appendix.







\section{Conclusion}

We have introduced in this paper a general ergodic framework for point process regression models, in which we were able to derive general asymptotic properties of both the Quasi Maximum Likelihood estimator and the Quasi Bayesian estimator. Because such formulation of ergodicity covers a wide range of models, we have shown that the Quasi Likelihood Analysis approach is robust in the sense that it applies to various contexts, provided that they encompass a family of laws of large numbers on the parametric stochastic intensity. \\
\medskip

Since, in practice, for Limit Order Books, parameters (and other variables of interest) are often estimated via such time average limits, we have decided to focus on various point process models that are directly involved in the modelling of the dynamics of LOB's. In particular, we have seen that Markovian models, self-exciting exponential Hawkes processes, and their mixtures that can be found in the literature are examples of applications of the QLA. \\
\medskip

 It is also important to know the limits of such procedure. Parametric estimations are undoubtebly powerfull when it comes to calibrating on real data a model that looks well suited to, say, the reproduction of stylized facts observed in practice. Nevertheless, in many cases the choice of a parametric model is quite arbitrary although it should be investigated with the help of mathematical tools. Information criteria (AIC, BIC), sparsity (e.g. LASSO type estimations), and hypothesis testing, are examples of model selection methods that can quantify in some sense to what extent models perform well (this last statement being deliberately vague here), even in mispecified contexts. Another approach consists in checking the validity of a model by testing the constancy of the parameters over time. In the same spirit as in \cite{potiron2015estimating}, but in a different asymptotic though, empirical tests could be conducted to investigate the stability of parameters. Should they fail, the choice of an ergodic model itself would have to be reconsidered. Although such considerations are beyond the scope of the present paper, they should be investigated in the future.\\
 \medskip

 Let us conclude this work by pointing out that in the litterature there is virtually no LOB model that allows random volumes of orders, except slight generalizations, either to the simple I.I.D case (see e.g. \cite{MuniTokeQueuing2013}), or in a theoretic Markovian framework \cite{huang2015ergodicity}. The lack of such model extensions is probably due to the fact that the very rich and complex structure of Order Books is still poorly understood. The QLA could be generalized to more intricate models for Limit Order Books in order to address the most general structure between random time events $T_i$ and random order sizes $V_i$. Accordingly, statistical inferences for marked point processes should thus be the subject of a further work as well. Along with the abovementioned model selection machinery, applications to real data analysis would certainly shed light on these little-known mechanisms. 

\section*{Acknowledgement}
This work was in part supported by CREST Japan Science and Technology Agency, Japan Society for the Promotion of Science : Grants-in-Aid for Scientific Research No. 24340015 (Scientific Research), No. 26540011 (Challenging Exploratory Research); NS Solutions Corporation; and by a Cooperative Research Program of the Institute of Statistical Mathematics. 

\section{Appendix}

\subsection{Proofs of Section 3}

\begin{proof}[Proof of Proposition \ref{extensionA3}]
Let us start with the identification of the mapping $\pi_\alpha(.,\theta)$ on $C_b(E,\reels)$. For the sake of simplicity we write $X_s^\alpha(\theta) = (\lambda^\alpha(s,\theta^{*}),\lambda^\alpha(s,\theta),\partial_\theta \lambda^\alpha(s,\theta))$, and we consider the sequence of probability measures on $E$ defined for any Borel subset $A \subset E$ :

				\beas
				\pi_{T,\alpha}^\theta(A) = \inv{T}\int_0^T{\proba[X_s^\alpha(\theta) \in A]ds}. 
				\eeas

				For any $\psi \in C_b(E,\reels)$, an immediate application of the dominated convergence theorem and [A3] shows that $\int_E{\psi(x)\pi_{T,\alpha}^\theta(dx)} \to \pi_\alpha(\psi,\theta)$. On the other hand, it is straightforward to see that the family $(\pi_{T,\alpha}^\theta)_{T \in \reels_+}$ is tight since the family $(X_t^\alpha(\theta))_{t \in \reels_+}$ is tight itself. Thus along a subsequence  $(\pi_{T,\alpha}^\theta)_{T \in \reels_+}$ converges weakly to some $\pi_\alpha^\theta$, and for any $\psi \in C_b(E,\reels)$, $\pi_\alpha(\psi,\theta) = \int_{E}{\psi(x)\pi_\alpha^\theta(dx)}$ which determines uniquely $\pi_\alpha^\theta$.\\

 Let us now show (i). Let $\psi \in C_\uparrow(E,\reels)$. Because we can separate $\psi$ as usual $\psi = \psi_+ - \psi_-$, where $\psi_+$ and $\psi_-$ are respectively the positive and negative parts of $\psi$, it is not restrictive to assume that $\psi$ is non-negative. Consider first the case where $\psi$ is continuous on the whole set $E$. Put then $\delta > 0$, and consider for $T \in \reels_+$

\beas
V_T^{\alpha}(\theta) = \frac{1}{T}\int_0^T{\psi(X_s^\alpha(\theta)) ds}
\eeas
and
\beas
V_T^{\delta,\alpha}(\theta) = \frac{1}{T}\int_0^T{\left(\psi(X_s^\alpha(\theta)) \wedge \delta\right) ds},
\eeas
so 

\beas
\esp \l|V_T^{\alpha}(\theta)-V_T^{\delta,\alpha}(\theta)\r| &\leq& \frac{1}{T}\esp \l[ \int_0^T{\l|\left(\psi(X_s^\alpha(\theta))-\delta\r)1_{\{\psi(X_s^\alpha(\theta))\geq \delta\}}\r|ds}\r]\\
&\leq& \frac{1}{T}\esp \l[\int_0^T{\psi(X_s^\alpha(\theta)) 1_{\{\psi(X_s^\alpha(\theta))\geq \delta\}}ds}\r]\\
&\leq& \frac{1}{T\delta} \esp \int_0^T{\psi(X_s^\alpha(\theta))^2ds} \\
&\leq& \frac{Q}{\delta},
\eeas
				
where $Q = \sup_{t \in \reels_+}{\esp \sup_{\theta \in \Theta} \psi(X_t^\alpha(\theta))^2}$. Now taking $T$ and $T^{'}$ positive real numbers, we get :
				
				\beas
				\esp|V_T^\alpha(\theta) -V_{T^{'}}^\alpha(\theta)| &\leq& \esp|V_T^\alpha(\theta) -V_{T}^{\alpha,\delta}(\theta)|+\esp|V_T^{\alpha,\delta}(\theta) -V_{T^{'}}^{\alpha,\delta}(\theta)|+\esp|V_{T^{'}}^\alpha(\theta) -V_{T^{'}}^{\alpha,\delta}(\theta)|\\
				&\leq& 2 \frac{Q}{\delta} +\esp|V_T^{\alpha,\delta}(\theta) -V_{T^{'}}^{\alpha,\delta}(\theta)|.\\
				\eeas
				
				Note that the second term can be controlled since the ergodicity assumption [A3] and the boundedness of moments of $V_T^{\alpha,\delta}(\theta)$ also imply the convergence in $\mathbb{L}^p$ of the LLN for any $p \geq 1$. For any $\epsilon >0$, taking large $\delta$ and sufficiently large $T$ and $T^{'}$ we can thus conclude that $\esp|V_T^\alpha(\theta) -V_{T^{'}}^\alpha(\theta)| \leq \epsilon$, showing that  $V_T^\alpha(\theta)$ is a Cauchy sequence in $T$, and thus converges to a finite value $\pi_\alpha(\psi,\theta)$. In particular, thanks to the monotone convergence theorem applied to the family $\psi \wedge \delta$ and the measure $\pi_\alpha^\theta$ we have also shown that the identity  $\pi_\alpha(\psi,\theta) =  \int_{E}{\psi(x)\pi_\alpha^\theta(dx)}$ holds. This shows that $\pi_\alpha(.,\theta)$ can be extended to $C_\uparrow(E,\reels)$.\\

				In the case where $\psi$ has singularities at points of the form $(0,v,w)$ and $(u,0,w)$, first construct a sequence of function $\psi_\eta \in C_b(E,\reels)$ such that $\psi_\eta \uparrow \psi$ when $\eta \to 0$ and such that for any $(u,v,w) \in E$, $\psi_\eta(u,v,w)=0$ whenever $u\leq \eta $ or $v\leq \eta $, and $\psi_\eta(u,v,w) =\psi_\eta(u,v,w)$ whenever $u\geq 2\eta $ and $v\geq 2\eta $ . Such construction is possible because $\psi$ is lower semi-continuous. It is then straightforward to show that $V_T^\alpha(\theta)$ is again a Cauchy sequence in $\mathbb{L}^1$ as previously.\\  
				
				We now turn to the uniform convergence in $\theta$. Once again, assume first for the sake of simplicity that $\psi$ is continuous on $E$. As $\overline{\Theta}$ is compact, it is sufficient to show that the family $\l(\theta \to V_T^\alpha(\theta)\r)_{T \in \reels_+}$ is equicontinuous for, say, the $\mathbb{L}^1$ norm. Since $\sup_{t \in \reels_+} \esp[ \sup_{\theta \in \Theta} |X_t^\alpha(\theta)| ] < \infty$, the family $(X_t^\alpha(\theta))_{t \in \reels_+,\theta \in \Theta}$ is tight. Consider thus $\epsilon >0$ and a compact $K \in E$ such that $\proba[X_t^\alpha(\theta) \notin K] < \epsilon$ for all $t$ and $\theta$. Let $\tilde{\psi}$ be a $C^1$ function with compact support such that $\|\psi -\tilde{\psi}\|_{\infty, K} < \epsilon$ and $\tilde{\psi} \leq 2\psi$ outside $K$, where $\|.\|_{\infty, K}$ is the infinite norm over $K$. In particular $\tilde{\psi}$ is k-Lipschitzian for some $k > 0$.  Then 

				\beas
				\esp|\psi(X_t^\alpha(\theta)) - \psi(X_t^\alpha(\theta^{'}))| &\leq& \esp|(\psi-\tilde{\psi})(X_t^\alpha(\theta))| +\esp|(\psi-\tilde{\psi})(X_t^\alpha(\theta^{'}))| + \esp|\tilde{\psi}(X_t^\alpha(\theta)) - \tilde{\psi}(X_t^\alpha(\theta^{'}))|,
				\eeas
				but 

				\beas
				\esp|(\psi-\tilde{\psi})(X_t^\alpha(\theta))| &\leq& \epsilon + \esp|(\psi-\tilde{\psi})(X_t^\alpha(\theta)) 1_{\{ X_t^\alpha(\theta) \notin K \}}|\\
				&\leq& \epsilon + \esp \l[\l|(\psi-\tilde{\psi})(X_t^\alpha(\theta))\r|^2\r]^{\half}\proba[X_t^\alpha(\theta) \notin K]^\half\\
			    &\leq& \epsilon + M\epsilon^\half,
				\eeas

				where $M = 3 \sup_{t \in \reels_+}\|\sup_{\theta \in \Theta} \psi(X_t^\alpha(\theta))\|_2$. On the other hand,

				\beas
				\esp|\tilde{\psi}(X_t^\alpha(\theta)) - \tilde{\psi}(X_t^\alpha(\theta^{'}))| &\leq& k\esp|X_t^\alpha(\theta) - X_t^\alpha(\theta^{'})| \\
				&\leq& k|\theta - \theta^{'}|\sup_{t \in \reels_+}\esp \l|\sup_{\theta \in \Theta}  |\partial_\theta X_t^\alpha(\theta)|\r|.
				\eeas

				Thus we deduce that by taking $\theta$ sufficiently close to $\theta^{'}$, 
				\beas
				\esp|V_T^\alpha(\theta)-V_T^\alpha(\theta^{'})| \leq 3\epsilon +M\epsilon^\half, 
				\eeas
				the bound being independent of $T$. Such equicontinuity along with the pointwise convergence ensures the uniform convergence. The case where $\psi$ has singularities is treated similarly using the tightness of the family $\l(X_t^\alpha(\theta), \frac{1_{\{\lambda^\alpha(t,\theta) \neq 0\}}}{\lambda^\alpha(t,\theta)}\r)_{t \in \reels_+,\theta \in \Theta}$. 
\end{proof}

The following lemma is a verification of the permutation rule of the symbols $\partial_\theta$ and $\int_0^T$.

\begin{lemma*}
	The process $\int_0^T{\textnormal{log}(\lambda^\alpha(s,\theta))d\tilde{N}_s^\alpha}$ is well defined and admits a derivative for any $\theta \in \Theta$ and we have :
	$$\partial_\theta \int_0^T{\textnormal{log}(\lambda^\alpha(s,\theta))d\tilde{N}_s^\alpha} =\int_0^T{\partial_\theta \textnormal{log}( \lambda^\alpha(s,\theta))d\tilde{N}_s^\alpha}.$$
	\label{deriv}
\end{lemma*}

\begin{proof}[Proof of Lemma \ref{deriv}]
		
		For the $\int_0^T{...dN_s}$ part the result is immediate since for almost any $\omega$ the integral is simply a finite sum of terms that are differentiable in $\theta$. For the $\int_0^T{...ds}$ part, we first need to show that $\int_0^T{\text{log}(\lambda^\alpha(s,\theta))\lambda^\alpha(s,\theta^{*})ds}$ is well defined a.s., but we have :
		$$ \esp \int_0^T{\left|\text{log}(\lambda^\alpha(s,\theta))\lambda^\alpha(s,\theta^{*})\right|ds}  <+\infty, $$
		using the majoration $|\text{log}(x)| \leq \left|\frac{1}{x}\right|+ |x-1|$ and [A2]. With a similar argument we can easily show that 
		$$ \int_0^T{\left|\partial_\theta \text{log}(\lambda^\alpha(s,\theta))\lambda^\alpha(s,\theta^{*})\right|ds} \leq \int_0^T{\left|\sup_{\theta \in \Theta}\partial_\theta \text{log}(\lambda^\alpha(s,\theta))\right|\lambda^\alpha(s,\theta^{*})ds} < +\infty \text{  } \proba \text{-a.s.}.$$

		Now, for a given $\omega$ such that these random variables are finite, we can conclude by applying the dominated convergence theorem to the function $s \to \frac{\lambda^\alpha(s,\theta)[\omega]-\lambda^\alpha(s,\theta_0)[\omega]}{\theta-\theta_0}$ when $\theta \to \theta_0$, with respect to the measure $\lambda^\alpha(s,\theta^{*})[\omega]ds$.   
		
\end{proof}

The following lemma is useful to prove the $\mathbb{L}^p$ boundedness of stochastic integrals with respect to the processes $N^\alpha$ or $\tilde{N}^\alpha$.
\begin{lemma*}
	let $f_t$ be a predictable process and consider the martingale $\tilde{N}_t = N_t - \Lambda_t(\theta^{*})$.  Then for any $p \geq 1$, $\alpha \in \mathbf{I}$, we have the following majoration for some constant $C_p$  :
	$$ \esp \left[ \left|\int_0^T{f_s  d\tilde{N}_s^\alpha }\right|^{2^p}\right] \leq C_p\esp \left[ \int_0^T{f_s^{2^p}  \lambda^\alpha(s,\theta^{*})ds} \right] + C_p\esp \left[ \left|\int_0^T{f_s^2  \lambda^\alpha(s,\theta^{*})ds} \right|^{2^{p-1}}\right].$$
	whenever the expectations are well defined.
	\label{lemmaMajoration}
\end{lemma*}

\begin{proof}[Proof of Lemma \ref{lemmaMajoration}]
	For $p=1$ it is sufficient to notice that we have :
	$$\esp \left[ \left|\int_0^T{f_s  d\tilde{N}_s^\alpha }\right|^{2}\right]=\esp \left[ \int_0^T{f_s^{2}  \lambda^\alpha(s,\theta^{*})ds} \right]$$
	For $p >1 $, notice that by the Burkholder-Davis-Gundy inequality, one can write :
	
	\begin{eqnarray*}
	\esp \left[ \left|\int_0^T{f_s  d\tilde{N}_s^\alpha }\right|^{2^p}\right]&\leq& D_p \esp\left[\left|\int_0^T{f_s^2  dN_s^\alpha }\right|^{2^{p-1}}\right]\\
	&\leq& 2^{p-1}D_p \esp\left[\left|\int_0^T{f_s^2  d\tilde{N}_s^\alpha }\right|^{2^{p-1}}\right]+2^{p-1}D_p \esp\left[\left|\int_0^T{f_s^2  \lambda^\alpha(s,\theta^{*})ds }\right|^{2^{p-1}}\right]\\
	\end{eqnarray*}
	
	And by induction, one gets for some constant $Q_p$ :
	
	$$\esp \left[ \left|\int_0^T{f_s  d\tilde{N}_s^\alpha }\right|^{2^p}\right] \leq Q_p \sum_{q=1}^{p}{\esp \left| \int_0^T{f_t^{2^q}\lambda^\alpha(t,\theta^{*}) dt}\right|^{2^{p-q}}}. $$
	
	Now we show that for any $q \in \{1,...,p \}$, 
	$$\left(\int_0^T{f_t^{2^q}\lambda^\alpha(t,\theta^{*}) dt}\right)^{2^{p-q}} \leq \text{max}\left\{\left(\int_0^T{f_t^2 \lambda^\alpha(t,\theta^{*})dt}\right)^{2^{p-1}},\int_0^T{f_t^{2^p}\lambda^\alpha(t,\theta^{*})dt}\right\}.$$
	
	If $\int_0^T{f_t^2 \lambda^\alpha(t,\theta^{*})dt} =0$, we trivially get $0\leq0$. On the other hand, assuming $\int_0^T{f_t^2 \lambda^\alpha(t,\theta^{*})dt} >0$, put 
	
	$$ g_t = \frac{f_t}{\left(\int_0^T{f_t^2 \lambda^\alpha(t,\theta^{*})dt}\right)^\frac{1}{2}}.$$
	
	Then $\mu(dt) = g_t^2 \lambda^\alpha(t,\theta^{*}) dt$ is a probability measure on $[0,T]$. We thus have, by Jensen's inequality :
	
	$$\left(\int_0^T{g_t^{2^q}\lambda^\alpha(t,\theta^{*}) dt}\right)^{2^{p-q}} =\left(\int_0^T{g_t^{2^q-2}\mu(dt)}\right)^{2^{p-q}} \leq \left(\int_0^T{g_t^{2^p}\lambda^\alpha(t,\theta^{*})dt}\right)^{\frac{2^{p-1}-2^{-q}}{2^{p-1}-1}}.$$
	
	Depending on whether $\int_0^T{g_t^{2^p}\lambda^\alpha(t,\theta^{*})dt} \geq 1$ or not, we can conclude 
	
	$$\left(\int_0^T{g_t^{2^q}\lambda^\alpha(t,\theta^{*}) dt}\right)^{2^{p-q}} \leq \text{max}\left\{1,\int_0^T{g_t^{2^p}\lambda^\alpha(t,\theta^{*})dt}\right\}.$$
	
	And from the expression of $g$ we finally get the overall result.
	
\end{proof}

\begin{proof}[Proof of Lemma \ref{lemmaQLA}]
 Without loss of generality we can assume that $p$ is of the form $p=2^q>n$. For (\ref{eqDelta}), it is sufficient to apply Lemma \ref{lemmaMajoration} to $f_s = \lambda^\alpha(s,\theta^{*})^{-1}\partial_\theta \lambda^\alpha(s,\theta^{*})1_{\{\lambda^\alpha(s,\theta^{*}) \neq 0\}}$, along with [B2]. Following the same reasoning as for the classical approach, to get (\ref{eqY}) it is sufficient to show that :

$$ \l\|T^\gamma \sup_{\theta \in \Theta} \frac{M_T^\alpha(\theta)}{T} \r\|_p \to 0, $$
and
$$  \l\|T^\gamma\sup_{\theta \in \Theta} |V_T(\theta)-\mathbb{Y}(\theta)| \r\|_p \to 0.$$ 
 
For the first part, once again it is sufficient to apply Sobolev's inequality and then Lemma \ref{lemmaMajoration} to get the convergence for any $p >n$. For the second part, we wish to use Sobolev's inequality once again, but we need first to show that $\mathbb{Y}$ is of class $C^1$ on $\Theta$. This is ensured by the fact that $\sup_{\theta \in \Theta} |\partial_\theta V_T(\theta)-\mathbb{U}(\theta)| \to^\proba 0$ for some $\mathbb{U}$ thanks to [B3]. Thus $\mathbb{Y}$ is $C^1$ and $\partial_\theta \mathbb{Y} = \mathbb{U}$. Now, using Sobolev's inequality, there exists some constant $A(\Theta, p)$ such that :

\beas
T^{\gamma p} \esp \l[ \sup_{\theta \in \Theta} |V_T(\theta) - \mathbb{Y}(\theta)|^p\r] &\leq& A(\Theta, p) T^{\gamma p}\l( \int_\Theta{\esp \l[ |V_T(\theta) - \mathbb{Y}(\theta)|^p\r]d\theta}+\int_\Theta{\esp \l[ |\partial_\theta V_T(\theta) - \partial_\theta \mathbb{Y}(\theta)|^p\r]d\theta}\r).\\
\eeas
for any $p >n$. Because of the ergodicity assumption we have :

$$T^{\gamma p} \sup_{\theta \in \Theta}\esp \l[  |V_T(\theta) - \mathbb{Y}(\theta)|^p\r] +T^{\gamma p} \sup_{\theta \in \Theta}\esp \l[  |\partial_\theta V_T(\theta) - \partial_\theta \mathbb{Y}(\theta)|^p\r]\to 0,$$
which finally shows (\ref{eqY}). (\ref{eqGamma}) is a simple application of [B3], and finally (\ref{eqPartial}) is a straightforward consequence of Sobolev's inequality and Burkholder-Davis-Gundy inequality.

\end{proof}

\begin{proof}[Proof of Lemma \ref{lemmaM1}]
We first show the implication [M1] $\implies$ [A3]. For $T \in \reels_+$, $(\psi, \theta) \in C_b(E,\reels) \times \Theta$, we write 

\beas
X_t^\alpha(\theta) =  (\lambda^\alpha(t,\theta^{*}),\lambda^\alpha(t,\theta),\partial_\theta \lambda^\alpha(t,\theta)), 
\eeas
and
$$V_T^\alpha(\theta) = \inv{T} \int_0^T{\psi(X_t^\alpha(\theta)) dt}.$$

Define $\pi_\alpha$ as above and then, evaluating the mean square error, we get 

\beas 
\esp[(V_T^\alpha(\theta)-\pi(\psi,\theta))^2] &=&\text{Var}[V_T^\alpha(\theta)] + \l(\esp[V_T^\alpha(\theta)-\pi_\alpha(\psi,\theta)]\r)^2\\
&\leq& \frac{1}{T^2}\iint_{[0,T]^2}{\text{Cov}[\psi(X_t^\alpha(\theta)),\psi(X_s^\alpha(\theta))]dsdt} +\frac{1}{T}\int_0^T{\l(\esp[\psi(X_t^\alpha(\theta))-\pi_\alpha(\psi,\theta)]\r)^2dt}.\\
\eeas 

Now, using the mixing property for the first term of the right hand side leads to 

\beas
\frac{1}{T^2}\iint_{[0,T]^2}{\text{Cov}[\psi(X_t^\alpha(\theta)),\psi(X_s^\alpha(\theta))]dsdt}&\leq&\frac{1}{T^2}\iint_{[0,T]^2}{\rho_{|t-s|}dsdt}\\
&\leq& \frac{1}{T^2}\iint_{[0,T]^2}{1_{\{|t-s| \geq \sqrt T\}}\rho_{|t-s|}dsdt}+\frac{\sup_{u\in \reels_+} \rho_u}{T^2}\iint_{[0,T]^2}{1_{\{|t-s| \leq \sqrt T\}}dsdt}.\\
\eeas

The first term tends to $0$ thanks to the mixing property, and the second term is of order $O\l(\inv{\sqrt T} \r)$. The covariance term thus tends to $0$ as $T \to +\infty$. The convergence of the bias term to $0$ is immediate given the convergence in distribution of the argument of $\psi$.
\medskip

It remains to show [M2] $\implies$ [B3]. It is sufficient to prove the $\mathbb{L}^2$ convergence since the convergence in $\mathbb{L}^p$ for any $p \geq 1$ will be a straightforward consequence of the boundedness of moments of the intensity. Write the rescaled bias-variance decomposition 

\beas 
\esp[T^{2\gamma}(V_T^\alpha(\theta)-\pi(\psi,\theta))^2] &=&T^{2\gamma}\text{Var}[V_T^\alpha(\theta)] + T^{2\gamma}\l(\esp[V_T^\alpha(\theta)-\pi_\alpha(\psi,\theta)]\r)^2\\
&\leq& T^{2\gamma -2}\iint_{[0,T]^2}{\text{Cov}[\psi(X_t^\alpha(\theta),\psi(X_s^\alpha(\theta)]dsdt} +T^{2\gamma-1}\int_0^T{\l(\esp[\psi(X_t^\alpha(\theta)-\pi_\alpha(\psi,\theta)]\r)^2dt}\\
&\leq& A_1(T,\theta) + A_2(T,\theta).
\eeas 

Thanks to [M2], we have uniformly in $\theta \in \Theta$ that $\esp[\psi(X_t^\alpha(\theta)-\pi_\alpha(\psi,\theta)] = o(t^{-\gamma})$, and thus the fact that $\sup_{\theta \in \Theta} A_2(T,\theta) \to 0$ is clear. For $A_1(T,\theta)$, take some $0< \delta <1$ to be defined later and write

\beas
A_1(T,\theta)&\leq&T^{2\gamma -2}\iint_{[0,T]^2}{\rho_{|t-s|}dsdt}\\
&\leq& T^{2\gamma -2}\iint_{[0,T]^2}{1_{\{|t-s| \geq T^\delta\}}\rho_{|t-s|}dsdt}+T^{2\gamma -2}\sup_{u\in \reels_+} \rho_u \iint_{[0,T]^2}{1_{\{|t-s| \leq T^\delta\}}dsdt}.\\
\eeas

A short computation leads to the two following dominations (uniformly in $\theta \in \Theta$) :
\beas 
T^{2\gamma -2}\iint_{[0,T]^2}{1_{\{|t-s| \geq T^\delta\}}\rho_{|t-s|}dsdt} &=& o\l(T^{2\gamma -\epsilon\delta}\r),\\
T^{2\gamma -2}\sup_{u\in \reels_+} \rho_u \iint_{[0,T]^2}{1_{\{|t-s| \leq T^\delta\}}dsdt} &=& O\l(T^{2\gamma + \delta -1} \r),
\eeas 
and consequently both terms can be controlled if there exists $\delta$ such that $2\gamma - \epsilon \delta \leq 0$ and $2\gamma +\delta -1 < 0$, that is, if $\frac{2\gamma}{\epsilon} < 1-2\gamma$. But this is exactly the above-stated condition on $\epsilon$.   
\end{proof}
\subsection{Proofs of Section 4}

\begin{proof}[Proof of Proposition \ref{stability}]
Only the third point remains to be proved. We first notice that, for $\alpha \in \mathbf{I}$,

 \beas
 \esp |\lambda^\alpha(t) - \bar{\lambda}^\alpha(t)| &=& \esp  \l|\sum_{\beta \in \mathbf{I}}\int_{0}^{t-}{h_{\alpha\beta}(t-s)dN_s^\beta}-\sum_{\beta \in \mathbf{I}}\int_{-\infty}^{t-}{h_{\alpha\beta}(t-s)d\bar{N}_s^\beta}\r|\\
 &\leq& \esp \l[ \sum_{\beta \in \mathbf{I}}\int_{0}^{t-}{h_{\alpha\beta}(t-s)d\l|N^\beta-\bar{N}^\beta\r|_s}\r]+\esp\l[\sum_{\beta \in \mathbf{I}}\int_{-\infty}^{0}{h_{\alpha\beta}(t-s)\bar{\lambda}^\beta(s)ds}\r]\\
 &\leq& \esp \l[ \sum_{\beta \in \mathbf{I}}\int_{0}^{t}{h_{\alpha\beta}(t-s)\l| \lambda^\beta(s) - \bar{\lambda}^\beta(s)\r|ds}\r]+\esp\l[\sum_{\beta \in \mathbf{I}}\int_{-\infty}^{0}{h_{\alpha\beta}(t-s)\bar{\lambda}^\beta(s)ds}\r],\\
 \eeas
 where $\l|N^\beta-\bar{N}^\beta\r|$ is the counting measure such that $\Delta \l|N^\beta-\bar{N}^\beta\r|_s =\l|\Delta N_s^\beta-\Delta\bar{N}_s^\beta\r|$. In \cite{BremaudStability1996}, the authors show that the expression of its compensator is $\l| \lambda^\beta(s) - \bar{\lambda}^\beta(s)\r|$. Let 
 \beas
 f_t^\alpha &=& \l\| \lambda^\alpha(t) - \bar{\lambda}^\alpha(t)\r\|_1
 \eeas
 and 

 \beas
 r_t^\alpha &=&\esp\l[\sum_{\beta \in \mathbf{I}}\int_{-\infty}^{0}{h_{\alpha\beta}(t-s)\bar{\lambda}^\beta(s)ds}\r].
 \eeas

Finally we define the convolution product between two functions as 
\beas
\phi  \ast \psi (t) = \int_0^t{\phi(t-s)\psi(s)ds} 
\eeas
and the convolution product between two multidimensional applications $U : \reels_+ \to \reels^{m,n}$ and $V : \reels_+ \to \reels^{n,q}$ as $U \ast V : \reels_+ \to \reels^{m,q} $ such that 
\beas
(U \ast V)_{i,j}(t) &=&\sum_{k =1}^n {U_{i,k} \ast V_{k,j}(t)}.
\eeas

The above equation can thus be rewritten in the following vector form :

\beas
f_t &\leq& r_t + h \ast f(t),  
\eeas
and by an immediate induction one gets for any $n \in \naturels$ that 

\beas
f_t &\leq& \sum_{k=0}^n {h^{\ast k} \ast r (t)}  + h^{\ast (n+1)} \ast f(t).  
\eeas

Now, for a multidimensional application $U : \reels_+ \to \reels^{m,n}$, write the norm 

\beas
\|U\|_{1,\text{Leb}} &=& \sum_{i=1}^{m}\sum_{j=1}^{n} \int_0^{+\infty}{|U_{i,j}(t)|dt}. 
\eeas

Since all $f^\alpha$'s are bounded, say, by some $K_f \in \reels_+$, it is immediate to see that 

\beas
|h^{\ast n} \ast f(t)| &\leq& K_f\|h^{\ast n} \|_{1,\text{Leb}} \\
&\leq& K_f \l|\Phi^n \r| \to 0
\eeas
as $\rho(\Phi) <1$. Thus we can define the process 

\beas
R &=& \sum_{k=0}^{+\infty}{h^{\ast k} \ast r }
\eeas
and notice that, first, $f_t \leq R_t$, and that $R$ is a solution to the Markov renewal equation

\beas
R = r + h \ast R.
\eeas

This type of equations is deeply studied in \cite{AsmussenAppliedProbability2003}, chapter VII. Let us define $\Phi_q$ as

\beas
\Phi_q &=& \int_0^{+\infty}{e^{qt}h(t)dt}.
\eeas

It is shown in Theorem 4.6 of \cite{AsmussenAppliedProbability2003}, chapter VII, that if $\Phi_0 = \Phi$ is irreducible, and if there exists a real $q$ such that $\rho(\Phi_q)=1$, then for every $\alpha \in \mathbf{I}$ 
\beas
f_t^\alpha &\leq& R_\alpha(t) = O\l(e^{-qt}\r).
\eeas

$\Phi$ is obviously irreducible since all its coefficients are positive. Now, as for $q$ sufficiently small $\Phi_q =\l[\frac{c_{\alpha\beta}}{a_{\alpha\beta} - q}\r] $, the application $q \to \rho(\Phi_q)$ is continuous, $\rho(\Phi_0) <1$ and $\rho(\Phi_q) \to +\infty$ as $q \to +\infty$ and thus Theorem 4.6 from \cite{AsmussenAppliedProbability2003} applies.
 \end{proof}

\begin{proof}[Proof of Proposition \ref{propVgeo}]
The Markovian property of the elementary excitations is a well-known result, see \cite{OakesMarkovianHawkes1975}. The $V$-geometric ergodicity has also been proved for a linear $V$, see \cite{AbergelHawkes2015}. The case of an exponential $V$ follows a very similar procedure. As in \cite{AbergelHawkes2015}, we need to apply Theorem 6.1 from \cite{MeynTweedieprocessiii1993}.\\

Define $\call$ the infinitesimal generator of $\cale$. If for some $V$, $c>0$, and $f<\infty$, the following drift criterion holds

\beas
\label{Vdrift}
\call V(\epsilon) &\leq& -cV(\epsilon)+f \text{, } \epsilon \in \reels_+^{d \times d},
\eeas
and if every compact set is petite (see \cite{MeynTweedieMarkovChain2009}), then Theorem 6.1 of \cite{MeynTweedieprocessiii1993} applies and $\cale$ is $V$-geometrically ergodic. We first show the drift condition, and postpone the second property to Lemma \ref{Tchain} below. Let us take some $M \in \reels_+^{d \times d}$ to be fixed later. It is not hard to see that the infinitesimal generator of $\cale$ has the following representation  

\beas
\call f(\epsilon) &=& \sum_{\alpha \in \mathbf{I}}{\lambda^\alpha}(f(\epsilon+\tilde{C}_\alpha)-f(\epsilon)) - \sum_{\alpha,\beta \in \mathbf{I}}{a_{\alpha\beta}\partial_{\epsilon_{\alpha \beta}}f(\epsilon)\epsilon_{\alpha\beta}},   
\eeas
with the obvious notation $\lambda^\beta = \nu_\beta + \sum_{\gamma \in \mathbf{I}}{\epsilon_{\beta\gamma}}$, and where $\tilde{C}_\alpha$ is the matrix whose elements are all zeros except for the $\alpha$-th column wich is the $\alpha$-th column of $C$. Applied to the $V$ of the proposition, we get 

\beas
\frac{\call V(\epsilon)}{V(\epsilon)} &=& \sum_{\beta \in \mathbf{I}}{\lambda^\beta}(e^{\sum_{\alpha \in \mathbf{I}} m_{\alpha\beta}c_{\alpha\beta} }-1) - \sum_{\alpha,\beta \in \mathbf{I}}{a_{\alpha\beta}m_{\alpha\beta}\epsilon_{\alpha\beta}}.   
\eeas

Consider now some continuous function $\xi$ such that $\xi(x)=o(x)$ on a vicinity of $0$ and such that the following Taylor expansion holds :

$$ e^{\sum_{\alpha \in \mathbf{I}} m_{\alpha\beta}c_{\alpha\beta} } - 1 \leq \sum_{\alpha \in \mathbf{I}} m_{\alpha\beta}c_{\alpha\beta}  + \xi(|M|).$$

We thus get 

\beas
\frac{\call V(\epsilon)}{V(\epsilon)} &\leq& \sum_{\beta \in \mathbf{I}}{\nu^\beta}\sum_{\alpha \in \mathbf{I}}{\l(m_{\alpha\beta}c_{\alpha\beta}  + \xi(|M|)\r)} +\sum_{\beta, \gamma \in \mathbf{I}}{\epsilon_{\beta\gamma} \sum_{\alpha \in \mathbf{I}}{\l(m_{\alpha\beta}c_{\alpha\beta}  + \xi(|M|)\r)}}- \sum_{\alpha,\beta \in \mathbf{I}}{a_{\alpha\beta}m_{\alpha\beta}\epsilon_{\alpha\beta}}.   
\eeas

Consider now $\kappa \in (\reels_+)^d$ an eigenvector of $\Phi^T$, the transposed matrix of $\Phi$, associated to $\rho(\Phi)$, and put for any $\alpha,\beta \in \mathbf{I}$

\beas
m_{\alpha\beta} &=& \frac{\kappa_{\alpha}}{a_{\alpha\beta}}.
\eeas

Such eigenvector exists thanks to Perron-Frobenius theorem. It is easy to see that we have then 

\beas
\frac{\call V(\epsilon)}{V(\epsilon)} &\leq& (\rho(\Phi)-1) \sum_{\beta,\gamma \in \mathbf{I}}{(\kappa_\beta + d\xi(|\kappa|)) \epsilon_{\beta\gamma}} + \sum_{\beta \in \mathbf{I}}{\nu_\beta\l(\kappa_\beta + d \xi(|\kappa|)\r)}.
\eeas

Let us fix $c>0$. Consider the set 

\beas
K &=& \l\{\epsilon \in \reels_+^{d^2} \l| (\rho(\Phi)-1) \sum_{\beta,\gamma \in \mathbf{I}}{(\kappa_\beta + \xi(|\kappa|)) \epsilon_{\beta\gamma}} + \sum_{\beta \in \mathbf{I}}{\nu_\beta\l(\kappa_\beta + d \xi(|\kappa|)\r)} \geq -c \r.\r\}.
\eeas

Since $\rho(\Phi) - 1 < 0$, $K$ is a compact set. We have thus :

\beas 
\call V(\epsilon) \leq -cV(\epsilon) + \sup_{y \in K} \l|\call V(y)\r| + c \sup_{y \in K} V(y),
\eeas

which is the desired drift condition.
\end{proof}

We now show in the following lemma that $\cale$ is a $\psi$-irreducible T-chain. See \cite{MeynTweedieMarkovChain2009}, Chapters 4 and 6, for more information about this concept.

\begin{lemma*}

For $t \in \reels_+$, let $\mathbf{Q}^t$ be the transition kernel associated to $\cale(t)$, i.e. for any $\epsilon \in \reels_+^{d \times d}$ and $A \in \mathbf{B}(\reels_+^{d \times d})$,
$$\mathbf{Q}^t(\epsilon,A)=\proba[\cale(t) \in A | \cale(0) = \epsilon].$$ 

$\cale$ verifies the following properties :

\bd
\im [(i)] $\cale$ is a T-chain, i.e. there exists a non-trivial kernel $\mathbf{T}$ such that for any $A \in \mathbf{B}(\reels_+^{d \times d})$, $\mathbf{T}(.,A)$ is lower semi-continuous, and 
$$\mathbf{Q}(.,A) \geq \mathbf{T}(.,A).$$ 
\im [(ii)] $\cale$ is $\psi$-irreducible.
\ed

Every compact set is thus a petite set.
\label{Tchain}
\end{lemma*}

We denote by $(T_n)_{n \in \naturels}$ the stopping times associated to the jumps of $N$. In other words, $T_k$ represents the time of the $k$-th jump of $N$, and we arbitrarily define $T_0=0$. We also write $\Delta T_n = T_n -T_{n-1}$ and finally $K_n$ for the label in $\mathbf{I}$ for which the jump has occured, i.e. $K_n$ is a random variable that takes values in $\mathbf{I}$ and that is the only $\beta \in \mathbf{I}$ such that $\Delta N_{T_n}^\beta = 1$. The following result is well-known :

\begin{lemma*}

for an initial state $\cale(0)=\epsilon$, write $\mu^\alpha(t,\epsilon) = \nu_\alpha+ \sum_{\beta \in \mathbf{I}}{\epsilon_{\alpha\beta}e^{-a_{\alpha\beta}t}}$, and $\mu(t,\epsilon) = \sum_{\alpha \in \mathbf{I}}{\mu^\alpha(t,\epsilon)}$.

Then :\\
\bd
\im [(i)] $\Delta T_1$ has the distribution density (with respect to Lebesgue measure) $f^{\Delta T_1}(t|\cale(0)=\epsilon)=f(t,\epsilon) = \mu(t,\epsilon)e^{-\int_0^t{\mu(s,\epsilon)ds}}$. 
 
\im [(ii)] Moreover, conditionnally to $\Delta T_1$, $K_1$ has the distribution $\proba[K_1 = \beta | \Delta T_1, \cale(0)=\epsilon] = \frac{\mu^\beta(\Delta T_1,\epsilon)}{\mu(\Delta T_1,\epsilon)}$.
\ed
\label{lawJumps}
\end{lemma*}

\begin{proof}[Proof of Lemma \ref{Tchain}]

By an immediate application of the strong Markovian property and Lemma \ref{lawJumps}, we easily get that for any $j \geq 1$,

\beas
f^{\Delta T_j}(t|\cale( T_{j-1})) = f(t,\cale( T_{j-1})),
\eeas
and
\beas
\proba[K_j = \beta | \Delta T_j, \cale(T_{j-1})] = \frac{\mu^\beta(\Delta T_j,\cale(T_{j-1}))}{\mu(\Delta T_j,\cale(T_{j-1}))},
\eeas
and thus, the joint density of $(\Delta T_1,K_1,...,\Delta T_n,K_n)$ given $\cale(0)=\epsilon$ can be written as 

\beas
f^{(\Delta T_1,K_1,...,\Delta T_n,K_n)}(t_1,k_1,...,t_n,k_n|\epsilon) &=& \frac{\mu^{k_n}(t_n,\bar{\epsilon}(t_1,k_1,...,t_{n-1},k_{n-1}|\epsilon))}{\mu(t_n,\bar{\epsilon}(t_1,k_1,...,t_n|\epsilon))}f(t_n,\bar{\epsilon}(t_1,...,t_{n-1},k_{n-1}|\epsilon))\times... \\
&...&\times \frac{\mu^{k_1}(t_1,\epsilon)}{\mu(t_1,\epsilon)}f(t_1,\epsilon),
\eeas
where $\bar{\epsilon}(t_1,k_1,...,t_j,k_j|\epsilon)$ is the value of $\cale(T_j)$ when $(\Delta T_1,K_1,...,\Delta T_n,K_n) = (t_1,k_1,...,t_j,k_j)$, i.e. 

\beas
\bar{\epsilon}_{\alpha\beta}(t_1,k_1,...,t_j,k_j|\epsilon) = e^{-(t_1+...+t_j)}\epsilon_{\alpha\beta} + \sum_{i \leq j|k_i = \beta }{c_{\alpha\beta}e^{-a_{\alpha\beta}(t_{i+1}+...+t_j)}}.
\eeas

Note that in particular both $\bar{\epsilon}(t_1,k_1,...,t_j,k_j|\epsilon)$ and $f^{(\Delta T_1,K_1...,\Delta T_n,K_n)}$ are $C^\infty$ in $(t_1,...,t_n)$ and in $\epsilon$. We now show that $\mathbf{Q}^t$ has a lower semi-continuous component $\mathbf{T}$. To do this, consider first the following domination, for $\epsilon \in \reels_+^{d \times d}$ and $A \in \mathbf{B}(\reels_+^{d \times d})$ :

\beas
\mathbf{Q}^t(\epsilon,A) &=& \proba[\cale(t) \in A | \cale(0) = \epsilon]\\
&\geq& \proba\l[\{\cale(t) \in A\} \cap \{\#\{j | T_j \leq t\} = d^2\} \cap \bigcap_{\beta \in \mathbf{I}}\l\{\Delta N_{T_{(\beta-1)d+i}}^\beta=1, i \leq d\r\} |\cale(0) = \epsilon\r].
\eeas

We thus select only events for which there are exactly $d^2$ jumps before $t$, and the first $d$ jumps occur on $N^1$, the next $d$ ones occur on $N^2$, and so on. In that case, the value of $\cale(t)$ is completely determined by $(t,T_1,...,T_{d^2})$, and has the form

\beas
\bar{\epsilon}_{\alpha\beta}(t,T_1,...,T_{d^2}|\epsilon) = \epsilon_{\alpha\beta}e^{-a_{\alpha\beta}t}+\sum_{(\beta-1)d \leq i \leq \beta d -1}{c_{\alpha\beta}e^{-a_{\alpha\beta}(t- T_1+...+T_i)} }.
\eeas

Again, $\epsilon \to \bar{\epsilon}(t,t_1,...,t^{d^2} |\epsilon)$ is $C^\infty$ in $\epsilon$. For conciseness, we write $\phi_n(t_1,...,t_n|\epsilon)$ for $f^{(\Delta T_1,K_1,...,\Delta T_n,K_n)}(t_1,k_1,...,t_n,k_n|\epsilon)$ with each $k_i$ equal to the unique $\beta$ such that $ (\beta-1)d \leq i \leq \beta d -1$ and $k_{d^2+1}$ is any arbitrary $\beta$ (it plays no role). We have thus

\beas
\mathbf{Q}^t(\epsilon,A) &\geq& \idotsint_{\reels_+^{d^2+1}}{1_{\{\bar{\epsilon}(t,t_1,...,t_{d^2}|\epsilon) \in A\}}1_{\{t_1+...+t_{d^2} < t\} \cap \{t_1+...+t_{d^2+1} > t\}}\phi_{d^2+1}(t_1,...,t_{d^2}|\epsilon)dt_1..dt_{d^2+1}}. 
\eeas

Since the indicator $1_{\{\bar{\epsilon}(t,t_1,...,t_{d^2}|\epsilon) \in A \}}$ is not lower semi-continous in $\epsilon$ whenever $A$ is not open, we need to remove the dependency of this term in $\epsilon$ by a change of variable. Consider the transformation 
\beas
\Psi_\epsilon : (t_1,...,t_{d^2}) \to \Psi_\epsilon(t_1,...,t_{d^2}) = \epsilon(t,t_1,...,t_{d^2}|\epsilon).
\eeas

$\Psi_\epsilon$ is clearly $C^\infty$ in its argument and in $\epsilon$. We write $J_{\Psi_\epsilon}(t_1,...,t_n) = \l|\text{det}\nabla \Psi_\epsilon(t_1,...,t_n)\r|$ the determinant of the Jacobian matrix of $\Psi_\epsilon$ at the point $(t_1,...,t_n)$. Tedious but straightforward computation leads to the following representation of the Jacobian :

\beas
J_{\Psi^\epsilon}(t_1,...,t_{d^2}) &=& \begin{vmatrix} M_1(t_1,...,t_{d}) & * & \cdots & * \\ 0 & M_2(t_1,...,t_{2d}) & \cdots&* \\ \vdots & \ddots & \ddots &\vdots\\ 0 & \cdots & \cdots & M_d(t_1,...,t_{d^2})
\end{vmatrix}\\
&=& \prod_{ i \in \mathbf{I}} J_i(t_1,...,t_{id}),
\eeas
where $J_k(t_1,...,t_{kd}) = \text{det} M_k(t_1,...,t_{kd})$. Note that in particular it does not depend on $\epsilon$. Thus it is sufficient to show that under some conditions, each $J_k \neq 0$.  Again, after elementary linear operations on the $M_k$'s, it is possible to represent $|J_k|$ as follows : 

\beas
J_k(t_1,...,t_{kd}) &=& \prod_{ i \in \mathbf{I}}{c_{ik}a_{ik}e^{-a_{ik}(t-t_1-...-t_{1+(k-1)d})}} \begin{vmatrix} 1& \cdots & \cdots & 1 \\ e^{a_{1k}t_{2+(k-1)d}} & \cdots & \cdots & e^{a_{dk}t_{2+(k-1)d}} \\ \vdots & \vdots & \vdots&\vdots \\ e^{a_{1k}(t_{2+(k-1)d}+...+t_{dk})} & \cdots & \cdots & e^{a_{1k}(t_{2+(k-1)d}+...+t_{dk})}
\end{vmatrix}.
\eeas

Telling if $J_k \neq 0$ is not easy for general $(t_1,...,t_{d^2})$. On the other hand, if there exists $k \in \mathbf{I}$ and two indices $i,j \in \mathbf{I}$ such that $a_{ik}=a_{jk}$, two columns of the determinant are equal and thus $J_k = 0$. $J_k$ also vanishes if one $c_{ij} = 0$. For now let's assume for simplicity that for any $k \in \mathbf{I}$, for any $i,j \in \mathbf{I}$, $a_{ik} \neq a_{jk}$, and $c_{ij} \neq 0$. \\

Note that if $t_1 = ... = t_{d^2} = t_0 >0$, each $J_k$ is a Vandermonde determinant, and thus 

\beas
J_k(t_0,...,t_0) &=& \prod_{i \in \mathbf{I}}{{c_{ik}a_{ik}e^{-a_{ik}(t-(1+(k-1)d)t_0)}}}\prod_{j \geq l}(e^{-a_{lk}t_0}-e^{-a_{jk}t_0}) \neq 0.
\eeas

By continuity, if all the $t_i$'s are close enough to each other, say, belong to some interval $]t_0 - \eta, t_0+\eta[$ for $\eta$ small enough and such that $t_0 -\eta >0$, we have $J_{\Psi_\epsilon} \neq 0$. Writing 

\beas
B(t_1,...,t_{d^2+1}) = \{(t_1,...,t_{d^2}) \in ]t_0-\eta,t_0+\eta[\} \cap \{t_1+...+t_{d^2} < t\} \cap \{t_1+...+t_{d^2+1} > t\},
\eeas
we have by change of variable

\beas
\mathbf{Q}^t(\epsilon,A) &\geq& \idotsint_{ \reels_+^{d^2+1}}{1_{\{\bar{\epsilon}(t,t_1,...,t_{d^2}|\epsilon) \in A\}}1_{\{B(t_1,...,t_{d^2+1})\}}\phi_{d^2+1}(t_1,...,t_{d^2+1}|\epsilon)dt_1...dt_{d^2+1}} \\
&=&\idotsint_{\Psi_\epsilon\l(\reels_+^{d^2}\r)\times \reels_+}{1_{\{(y_1,...,y_{d^2}) \in A\}}1_{\{B(\Psi_\epsilon^{-1}(y_1,...,y_{d^2}),t_{d^2+1})\}}\cdots}\\
&\cdots&\phi_{d^2+1}\l(\Psi^{-1}(y_1,...,y_{d^2}),t_{d^2+1}|\epsilon\r)J_{\Psi_\epsilon}\l(\Psi_\epsilon^{-1}(y_1,...,y_{d^2})\r)^{-1}dy_1...dy_{d^2}dt_{d^2+1}.
\eeas

Finally, it is easy to see that $B(\Psi_\epsilon^{-1}(y_1,...,y_{d^2}),t_{d^2+1})$ is a countable union of open intervals whose boundaries depend continuously on $\epsilon$, and $1_{B(\Psi_\epsilon^{-1}(y_1,...,y_{d^2}),t_{d^2+1})}$ is thus lower semi-continuous in $\epsilon$. Because $\phi_{d^2+1}$, $\Psi^{-1}$ and $J_\Psi$ are continuous in their component and in $\epsilon$, this finally shows that $\mathbf{Q}^t$ has a non-trivial lower semi-continuous component $\mathbf{T}$, and $\cale$ is a T-chain. \\
\medskip

Since $0$ is trivially a reachable point for $\cale$,  Proposition 6.2.1 in \cite{MeynTweedieMarkovChain2009} implies that $\cale$ is $\psi$-irreducible. Thus Theorem 6.2.5 (ii) in  \cite{MeynTweedieMarkovChain2009} implies that every compact set is petite, and we are done.

\end{proof}

\begin{remark} \rm
If there exists $k, i, j \in \mathbf{I}$ such that $a_{ik}=a_{jk}= a$ the kernel of $\cale$ is degenerate. On the other hand, one can easily see that, if $c_{jk} \neq 0$,  

$$\epsilon_{ik}(t) = c_{ik}\int_0^t{e^{-a(t-s)}dN_s^k} = \frac{c_{ik}}{c_{jk }}\epsilon_{jk}(t).$$ 

Thus by reducing the dimension of the state space of $\cale$, i.e. putting $\tilde{\cale} = \l[\epsilon_{\alpha\beta}\r]_{\alpha\beta \in \mathbf{I}^2-{(i,k)}}$, we see that $\tilde{\cale}$ is obviously still Markovian. We can also do this operation whenever some $c_{ij}$ is null. $\tilde{\cale}$ is then still a Markovian process but verifies the non-degeneracy conditions on $A$ and $C$. In practice we can therefore assume the non-degeneracy condition, since the ergodicity of the reduced process implies obviously the ergodicity of $\cale$.
\end{remark}



We finally turn to the verification of [B2], [M2], and [B4] to complete the proof of Theorem \ref{thmHawkesQLA}.

\begin{lemma*}
	The parametric exponential Hawkes model verifies \textnormal{[B2]}. 
\end{lemma*}

\begin{proof}
Points (ii) and (iii) are immediate since for any $t \in \reels_+$ and $\alpha \in \mathbf{I}$ we have $\lambda^\alpha(t,\theta) \geq \underline{\nu} >0$.
Before showing (i), note that any component of $\partial_\theta^i \lambda(t,\theta)$, ($i \leq 3$) can be expressed as a linear combination of the following terms :

\beas
\nu_\alpha &\leq& \bar{\nu},\\
\int_0^{t-}{c_{\alpha\beta}(t-s)^j e^{-a_{\alpha\beta}(t-s)}dN_s^\beta} &\leq& \int_0^{t-}{\bar{c}(t-s)^j e^{-\underline{a}(t-s)}dN_s^\beta},\\
\int_0^{t-}{(t-s)^j e^{-a_{\alpha\beta}(t-s)}dN_s^\beta} &\leq& \int_0^{t-}{(t-s)^j e^{-\underline{a}(t-s)}dN_s^\beta},\\
\eeas
for $j \leq 3$. It is thus sufficient to show that for any $p >0$, $\alpha,\beta \in \mathbf{I}$, 

\beas
\sup_{t \in \reels_+} \esp \l| \int_0^{t-}{(t-s)^j e^{-\underline{a}(t-s)}dN_s^\beta} \r|^p &<& +\infty.
\eeas

Without loss of generality we may assume that $p = 2^k$, $k \geq 1$. Using Lemma \ref{lemmaMajoration} again, we have for some constant $B$ depending on $k$ only, 

\beas
\esp \l| \int_0^{t-}{(t-s)^j e^{-\underline{a}(t-s)}dN_s^\beta} \r|^{2^k} &\leq&B\esp  \int_0^{t-}{(t-s)^{2^kj} e^{-2^k\underline{a}(t-s)}\lambda^\beta(s,\theta^{*})ds} +B\esp \l| \int_0^{t-}{(t-s)^{2j} e^{-2\underline{a}(t-s)}\lambda^\beta(s,\theta^{*})ds} \r|^{2^{k-1}} \\
&+&B\esp \l| \int_0^{t-}{(t-s)^j e^{-\underline{a}(t-s)}\lambda^\beta(s,\theta^{*})ds} \r|^{2^{k}} \\
&=& I + II + III.\\
\eeas

Since for any $q \in \naturels$, $\esp [\lambda(s,\theta^{*})^q]$ converges to a finite value whenever $s \to +\infty$ (it is an immediate consequence of the $V$-geometric ergodicity shown in Proposition \ref{propVgeo}), for all $\beta \in \mathbf{I}$, we bound from above $\esp [\lambda^\beta(s,\theta^{*})^q]$  by some common value $m_q$. We also write $M(m,r) = \int_0^{+\infty}{u^m e^{-ru}du} < +\infty$ for every $ m \in \naturels$ and $r >0$. To control $I$, write : 

\beas
I &\leq& B m_1  \int_0^{t}{(t-s)^{2^kj} e^{-2^k\underline{a}(t-s)}ds}\\
&\leq& Bm_1 M(2^kj,2^k\underline{a}).
\eeas

To deal with $II$, put $q = 2^{k-1}$. If $t=0$, we trivially have $II = 0$. When $t >0$, by Jensen's inequality applied to the probability measure $\mu(ds)= \frac{e^{2\underline{a}s}ds}{\int_0^t{e^{2\underline{a}s}ds}}$ on $[0,t]$, we have 

\beas
II &=&\esp \l| \int_0^{t-}{(t-s)^{2j} e^{-2\underline{a}(t-s)}\lambda^\beta(s,\theta^{*})ds} \r|^{q}\\
&\leq& \l(\int_0^t{e^{2\underline{a}s}ds}\r)^{q} \esp \l| \int_0^{t-}{(t-s)^{2qj} e^{-2\underline{a}qt}\l(\lambda^\beta(s,\theta^{*})\r)^q\mu(ds)} \r|\\
&\leq&\l(\int_0^t{e^{2\underline{a}s}ds}\r)^{q-1}e^{-2\underline{a}(q-1)t}\esp \l| \int_0^{t-}{(t-s)^{2qj} e^{-2\underline{a}(t-s)}\l(\lambda^\beta(s,\theta^{*})\r)^qds} \r|\\
&\leq& \inv{(2\underline{a})^{q-1}}M(2qj,2\underline{a})m_q.
\eeas


Finally $III$ is dominated as $II$.
\end{proof}

\begin{lemma*}
The parametric exponential Hawkes model verifies \textnormal{[M2]}.
\label{hawkesB3}
\end{lemma*}
 \begin{proof}[Proof of Lemma \ref{hawkesB3}]
Let us deal with the mixing condition first. Consider two elements $\phi, \psi$ of $D_\uparrow(E,\reels)$. Note that since the exponential Hawkes process intensities are bounded from below it is sufficient to consider the case where $\phi, \psi$ and their derivatives are of polynomial growth in their arguments. Define thus for $\alpha \in \mathbf{I}$,

\beas
X^\alpha(t,\theta) &=& (\lambda^\alpha(t,\theta^{*}),\lambda^\alpha(t,\theta),\partial_\theta \lambda^\alpha(t,\theta)),
\eeas 
and its truncation at $s \leq t$ :

\beas
\tilde{X}^\alpha(s,t,\theta)&=& \l(\lambda^\alpha(t,\theta^{*}),\sum_{\beta \in \mathbf{I}}\int_s^{t–}{h_{\alpha\beta}(t-u,\theta)dN_u^\beta},\sum_{\beta \in \mathbf{I}}\int_s^{t–}{\partial_\theta h_{\alpha\beta}(t-u,\theta)dN_u^\beta}\r).
\eeas

We then write, for some $s \leq t-u$,

\beas
\l|\text{Cov}[\phi(X(t,\theta)),\psi(X(t+u,\theta))]\r| &\leq& \l|\text{Cov}[\phi(X(t,\theta)),\psi(\tilde{X}(t+s,t+u,\theta))]\r| \\
&+& \l|\text{Cov}\l[\phi(X(t,\theta)),\nabla \psi(\xi(t+s,t+u,\theta))\l(X(t+u,\theta)-\tilde{X}(t+s,t+u,\theta)\r)\r]\r|\\
&=& A_1(\theta)+A_2(\theta),
\eeas
where $\xi(t+s,t+u,\theta) \in \l[\tilde{X}(t+s,t+u,\theta),X(t+u,\theta)\r]$. It is easy to see that, given the exponential kernel of the intensities, there exists $\tilde{a}$, such that for any $m \geq 1$,

\beas
\sup_{\theta \in \Theta} \esp \l|X(t+u,\theta)-\tilde{X}(t+s,t+u,\theta)\r|^2 &=& O(e^{-\tilde{a}(u-s)}).
\eeas

Therefore, as $\psi$, $\nabla \psi$ and $\phi$ are of polynomial growth, this leads to 

\beas
\sup_{\theta \in \Theta} A_2(\theta) &=& O(e^{-\tilde{a}(u-s)})
\eeas

For a random variable $\phi(Z)$, we now write $\hat{\phi}(Z)$ for $\phi(Z) - \esp[\phi(Z)]$. We thus have 

\beas
A_1(\theta) &=& \l|\esp \l[\hat{\phi}(X(t,\theta))\hat{\psi}(\tilde{X}(t+s,t+u,\theta))\r]\r| \\
&=&　　\l|\esp \l[\hat{\phi}(X(t,\theta))\esp\l[\hat{\psi}(\tilde{X}(t+s,t+u,\theta)) \l|\r. \calf_{t+s}^{\cale}\r]\r]\r| \\
&=&\l|\esp \l[\hat{\phi}(X(t,\theta))\esp\l[\hat{\psi}(\tilde{X}(t+s,t+u,\theta)) \l|\r. \cale_{t+s}\r]\r]\r|, \\
\eeas
where we have applied the strong Markovian property to the process $\hat{\psi}(\tilde{X}(t+s,t+u,\theta))$ which is measurable with respect to the $\sigma$-algebra $\sigma\{\cale_v | t+s \leq v \leq t+u\}$. As $\esp\l[\hat{\psi}(\tilde{X}(t+s,t+u,\theta)) \l|\r. \cale_{t+s}\r]$ is a function of $\cale_{t+s}$, $u-s$ and $\theta$, we denote it by $F(\cale_{t+s},u-s,\theta)$. Note that $F(\cale_{t+s},u-s,\theta)$ is centered.  We thus get 

\beas
A_1(\theta) &=&\l| \esp \l[\hat{\phi}(X(t,\theta)) F(\cale_{t+s},u-s,\theta)\r]\r| \\
&=& \l|\esp \l[\hat{\phi}(X(t,\theta))\esp\l[F(\cale_{t+s},u-s,\theta) | \calf_{t}^{\cale}\r]\r]\r| \\
&=& \l|\esp \l[\hat{\phi}(X(t,\theta))\esp\l[F(\cale_{t+s},u-s,\theta) | \cale_{t}\r]\r] \r|,\\
\eeas
using once again the strong Markovian property. Again, because of the poynomial growth of $\psi$ and the uniform boundedness of the moments of $\cale
_v$ in $v \in \reels_+$, one can easily check that  $F(\cale_{t+s},u-s,\theta)$ can be bounded from above uniformly in $u-s$ and $\theta$ by some measurable random variable $F(\cale_{t+s})$ independent of $u-s$ and $\theta$, and dominated by $\sqrt{V}$ up to a multiplication by a constant factor. Finally, as $\cale$ is also $\sqrt{V}$-geometrically ergodic (see \cite{MeynTweedieMarkovChain2009}, lemma 15.2.5), we have that, since $F(\cale_{t+s},u-s,\theta)$ is centered,

\beas 
\l|\esp\l[F(\cale_{t+s},u-s,\theta) | \cale_{t}\r]\r| &\leq& R r^s\sqrt{V(\cale_t)}. 
\eeas

From here, we can conclude for $A_1(\theta)$ that

\beas
A_1(\theta) &\leq&  R r^s\esp \l[\l|\hat{\phi}(X(t,\theta))\r|\sqrt{V(\cale_t)}\r]
\eeas
and thus, using Cauchy-Schwarz inequality and the uniform boundedness of  $\esp \l[V(\cale_t)\r]$, (it is an easy consequence of Dynkin's formula, see e.g. \cite{MeynTweedieprocessiii1993}), we get that 

\beas 
\sup_{\theta \in \Theta} A_1(\theta) \leq R r^s \sup_{\theta \in \Theta}\esp \l[\l|\hat{\phi}(X(t,\theta))\r|^2\r]^\half \esp[V(\cale_t)]^{\half},
\eeas
thus
\beas
A_1(\theta) &=&  O(r^s).
\eeas

Setting $s = \frac{u}{2}$, we finally get that there exists some $\tilde{r} <1$ such that 

\beas
\sup_{\theta \in \Theta} \l|\text{Cov}[\phi(X(t,\theta)),\psi(X(t+u,\theta))]\r| &=& O(\tilde{r}^{u}),
\eeas
which shows the mixing condition with an exponential rate. We now turn to the stability condition. As in [M2], for a process $Y$, $\bar{Y}$ designates its stationary version if it exists.  We have 

\beas
t^\gamma \esp\l[\l|\phi(X(t,\theta)) - \phi(\bar{X}(t,\theta)) \r|\r] &\leq& t^\gamma\esp\l[\l|\nabla \phi(\xi(t,\theta))\r| \l|X(t,\theta) - \bar{X}(t,\theta) \r|\r] 
\eeas
for some $\xi(t,\theta) \in [X(t,\theta),\bar{X}(t,\theta)]$. The right hand side can be split as follows : 

\beas
t^\gamma\esp\l[\l|\nabla \phi(\xi(t))\r| \l|X(t,\theta) - \bar{X}(t,\theta) \r|\r] &\leq& t^\gamma\esp\l[\l|\nabla \phi(\xi(t))\r| \l(1_{\{\nabla \phi(\xi(t)) \leq  t^{1+ \gamma}\}}+1_{\{\nabla \phi(\xi(t)) >  t^{1+ \gamma}\}}\r)\l|X(t,\theta) - \bar{X}(t,\theta) \r|\r]\\
&\leq& t^{1+2\gamma}  \l\| |X(t,\theta) - \bar{X}(t,\theta)| \r\|_1 + t^{-1}\esp\l[\l|\nabla \phi(\xi(t))\r|^3  \l|X(t,\theta) - \bar{X}(t,\theta) \r| \r].\\
\eeas

Finally, thanks to Proposition \ref{stability} (iii), and because each $X^\alpha(t,\theta) - \bar{X}^\alpha(t,\theta)$ can be represented as a sum of integrals with respect to the measures $dN^\beta - d\bar{N}^\beta$ for $\beta \in \mathbf{I}$,  $\l\| |X(t,\theta) - \bar{X}(t,\theta)| \r\|_1$ is exponentially decreasing uniformly in $\theta$ and thus the first term on the right hand side tends to zero. For the second term, it is sufficient to notice that the expectation is uniformly bounded in $t$ and $\theta$. 

\end{proof}
\begin{lemma*}
 The exponential Hawkes model statisfies the non-degeneracy condition \textnormal{[B4]}. 
 \end{lemma*}

 \begin{proof}
 We first show that whenever $\theta \neq \theta^{*}$, $\mathbb{Y}(\theta) \neq 0$. consider thus $\theta \in \Theta$ and assume that $\mathbb{Y}(\theta) = 0$. As 
 \beas
 -\mathbb{Y}(\theta) &=& \sum_{\alpha \in \mathbf{I}}\esp \left[\bar{\lambda}^\alpha(t,\theta)-\bar{\lambda}^\alpha(t,\theta^{*})-\text{log}\frac{\bar{\lambda}^\alpha(t,\theta)}{\bar{\lambda}^\alpha(t,\theta^{*})}\bar{\lambda}^\alpha(t,\theta^{*})\right],
 \eeas
 and, for any $\alpha \in \mathbf{I}$, 
 \beas
 \bar{\lambda}^\alpha(t,\theta)-\bar{\lambda}^\alpha(t,\theta^{*})-\text{log}\frac{\bar{\lambda}^\alpha(t,\theta)}{\bar{\lambda}^\alpha(t,\theta^{*})}\bar{\lambda}^\alpha(t,\theta^{*}) &\geq& 0,
 \eeas
we thus have 

 \beas
 \bar{\lambda}^\alpha(t,\theta)-\bar{\lambda}^\alpha(t,\theta^{*})-\text{log}\frac{\bar{\lambda}^\alpha(t,\theta)}{\bar{\lambda}^\alpha(t,\theta^{*})}\bar{\lambda}^\alpha(t,\theta^{*}) &=& 0 \text{   } \proba\text{-a.s.},
 \eeas
 and it is easy to see that this in turn implies that for any $t \in \reels_+$, and any $\alpha \in \mathbf{I}$,

 \beas
 \bar{\lambda}^\alpha(t,\theta)&=&\bar{\lambda}^\alpha(t,\theta^{*}) \text{   } \proba\text{-a.s.}
 \eeas

 Now, as the trajectories of $\lambda(.,\theta)$ are left continuous for any $\theta \in \Theta$, we can conclude that $\proba\text{-a.s.}$, the whole trajectories of $\lambda(.,\theta)$ and $\lambda(.,\theta^{*})$ coincide. We have thus, $\proba\text{-a.s.}$, 

 \bea
 \nu_\alpha - \nu_\alpha^{*} &=& \sum_{\beta \in \mathbf{I}}\int_{-\infty}^t{\l(c_{\alpha\beta}^{*}e^{-a_{\alpha\beta}^{*}(t-s)}-c_{\alpha\beta}e^{-a_{\alpha\beta}(t-s)}\r)d\bar{N}_s^\beta},
 \label{eqNu}
 \eea
but as the left hand side is constant, the right hand side, which is a jump process, must have jumps of size zero, and thus $C = C^{*}$. It is then easy to see that the right hand side is now $C^1$ in the variable $t$, and its derivative must be zero. Taking the derivative on both side this gives 

 \beas 
 0 &=& \sum_{\beta \in \mathbf{I}}\int_{-\infty}^t{\l(c_{\alpha\beta}^{*}a_{\alpha\beta}^{*}e^{-a_{\alpha\beta}^{*}(t-s)}-c_{\alpha\beta}a_{\alpha\beta}e^{-a_{\alpha\beta}(t-s)}\r)d\bar{N}_s^\beta}.
 \eeas

Provided that the $c_{\alpha\beta}$ are non zeros, we get that $A = A^{*}$. Finally the right side of (\ref{eqNu}) vanishes and $\nu = \nu^{*}$. We now turn to the proof of the second part of the lemma. [B4] holds true if the Fisher information matrix $\Gamma = \sum_{\alpha \in \mathbf{I}}\esp \l[\frac{1}{\bar{\lambda}^\alpha(t,\theta^{*})} \partial_\theta \bar{\lambda}^\alpha(t,\theta^{*})\partial_\theta \bar{\lambda}^\alpha(t,\theta^{*})^T\r]$ is positive definite. Let us rewrite the elements of $\theta$ in the following order : First $\begin{pmatrix} a_{11} & \cdots & a_{1d} & c_{11} & \cdots & c_{1d} & \nu_1 \end{pmatrix}$, the parameters involved in the writing of $\bar{\lambda}^1(.,\theta^{*})$, then those involved in the writing of $\bar{\lambda}^2(.,\theta^{*})$ and so on. Fix thus $\textbf{x} = \begin{pmatrix} \textbf{x}_1 & \cdots & \textbf{x}_d \end{pmatrix} \in (\reels_+^{2d+1})^d$, such that $\textbf{x}^T \Gamma \textbf{x} = 0$, and let us show that $\textbf{x} = 0$. As the expression of $\Gamma$ is a block diagonal matrix, we have for any $t \in \reels_+$ that 

 \beas
0 &=& \textbf{x}^T \Gamma \textbf{x} \\
&=& \sum_{\alpha \in \mathbf{I}}{\textbf{x}_\alpha^T \esp \l[\frac{1}{\bar{\lambda}^\alpha(t,\theta^{*})} \partial_\theta \bar{\lambda}^\alpha(t,\theta^{*})\partial_\theta \bar{\lambda}^\alpha(t,\theta^{*})^T\r] \textbf{x}_\alpha} \\
&=& \sum_{\alpha \in \mathbf{I}}{\esp \l[\frac{1}{\bar{\lambda}^\alpha(t,\theta^{*})} \l(\partial_\theta \bar{\lambda}^\alpha(t,\theta^{*})\textbf{x}_\alpha \r)^2 \r]}. 
 \eeas

As for the first part of the proof, it implies that for any $\alpha \in \mathbf{I}$ the following process is null almost surely :

\beas
\partial_\theta \bar{\lambda}^\alpha(.,\theta^{*})\textbf{x}_\alpha = 0   \text{   } \proba\text{-a.s.}
\eeas

In particular it admits derivatives and those derivatives have no jumps. Writing $\textbf{x}_\alpha = (\textbf{x}_{\alpha,a},\textbf{x}_{\alpha,c},\textbf{x}_{\alpha,\nu})$, by an immediate induction we get for any $n \geq 1$ and any $\alpha, \beta \in \mathbf{I}$

\beas
nc_{\alpha\beta}a_{\alpha\beta}^{n-1}\textbf{x}_{\alpha,a}^\beta- a_{\alpha\beta}^n\textbf{x}_{\alpha,c}^\beta &=& 0, 
\eeas
and thus

\beas
c_{\alpha\beta}\textbf{x}_{\alpha,a}^\beta- \frac{a_{\alpha\beta}}{n}\textbf{x}_{\alpha,c}^\beta  &=& 0.
\eeas

Taking $n \to +\infty$, we get $\textbf{x}_{\alpha,a}^\beta = 0 $ and thus $\textbf{x}_{\alpha,c}^\beta = 0$. Finally, this implies that 

\beas 
0 = \partial_\theta \bar{\lambda}^\alpha(.,\theta^{*})\textbf{x}_\alpha = \textbf{x}_{\alpha,\nu}, 
\eeas
and thus $\textbf{x}_\alpha =0 $. 

 \end{proof}
\bibliography{biblio}
\bibliographystyle{abbrv} 

\end{document}

%% file: nakamacro27.tex
\newtheorem{remark}{Remark}

\def\bd{\begin{description}}
\def\ed{\end{description}}

\def\D2{\bbD_{2,\infty-}}

\def\D{{\bf D}}

\def\F{{\bf F}}

\def\K{{\bf K}}

\def\calb{{\cal B}}

\def\cale{{\cal E}}
\def\calf{{\cal F}}
\def\calg{{\cal G}}

\def\call{{\cal L}}

\def\cals{{\cal S}}

\def\half{\frac{1}{2}}

\def\be{\begin{equation}}
\def\ee{\end{equation}}
\def\bea{\begin{eqnarray}}
\def\eea{\end{eqnarray}}
\def\beas{\begin{eqnarray*}}
\def\eeas{\end{eqnarray*}}
\def\bi{\begin{itemize}}
\def\ei{\end{itemize}}
\def\im{\item}
\def\bd{\begin{description}}
\def\ed{\end{description}}
\def\l{\left}
\def\r{\right}

\newcommand{\bbD}{{\mathbb D}}

\newcommand{\bbR}{{\mathbb R}}

\newcommand{\reels}{\mathbb{R}}
\newcommand{\naturels}{\mathbb{N}}
\newcommand{\relatifs}{\mathbb{Z}}

\newcommand{\esp}{\mathbb{E}}
\newcommand{\proba}{\mathbb{P}}

\newcommand{\inv}[1]{\frac{1}{#1}}